\documentclass[11pt]{amsart}
\usepackage[utf8]{inputenc}
\usepackage{bm}
\usepackage{fontenc}
\usepackage{amsfonts}
\usepackage{amssymb}
\usepackage{amsmath}
\usepackage{amsthm}\usepackage{mathtools}
\usepackage{enumerate}
\usepackage{enumitem} 
\usepackage[pagebackref,colorlinks,linkcolor=blue,citecolor=blue,urlcolor=blue,hypertexnames=true]{hyperref}
\usepackage[pagebackref,hypertexnames=true]{hyperref}
\usepackage{enumitem}
\usepackage{mathrsfs}
\usepackage{tikz}
\usetikzlibrary{calc}
\usepackage{marginnote}
\usepackage{xcolor,enumitem}
\usepackage{soul}\usepackage{tikz}
\usepackage[all,cmtip]{xy}
	
\newcommand{\N}{\mathbb{N}}
\newcommand{\C}{\mathbb{C}}

\newcommand{\cV}{\mathcal{V}}

\newcommand{\cE}{\mathcal{E}}

\newcommand{\eps}{\varepsilon}

\newcommand{\Hawaii}{Hawai\kern.05em`\kern.05em\relax i}

\newcommand{\SOTh}{\mathrm{SOT}\text{-}}

\newcommand{\cstu}{\mathrm{C}^*_u}

\newcommand{\roeq}{\mathrm{Q}^*_u}

\newtheorem*{rigprob*}{Rigidity Problem for uniform Roe Algebras}
\newtheorem*{rigprobcorona*}{Rigidity Problem for uniform Roe Coronas}

\newcommand{\cstar}{$\mathrm{C}^*$}

\newcommand{\cU}{\mathcal{U}}
\newcommand{\cZ}{\mathcal{Z}}
\newcommand{\cF}{\mathcal{F}}
\newcommand{\cP}{\mathcal{P}}

\newcommand{\cB}{\mathcal{B}}
\newcommand{\cK}{\mathcal{K}}

\newcommand{\R}{\mathbb{R}}

\numberwithin{equation}{section}

\newtheorem{theorem}{Theorem}[section]
\newtheorem*{theorem*}{Theorem}
\newtheorem{proposition}[theorem]{Proposition}

\newtheorem*{proposition*}{Proposition}
\newtheorem{lemma}[theorem]{Lemma}
\newtheorem*{lemma*}{Lemma}
\newtheorem{corollary}[theorem]{Corollary}
\newtheorem*{corollary*}{Corollar}

\newtheorem*{fact*}{Fact}
\theoremstyle{definition}
\newtheorem{definition}[theorem]{Definition}
\newtheorem*{definition*}{Definition}

\newtheorem*{acknowledgments}{Acknowledgments}
\newtheorem{claim}[theorem]{Claim}
\newtheorem*{claim*}{Claim}

\newtheorem*{conjecture*}{Conjecture}

\newtheorem{assumption}[theorem]{Assumption}

\theoremstyle{remark}

\newtheorem*{example*}{Example}
\newtheorem{remark}[theorem]{Remark}
\newtheorem*{remark*}{Remark}

\newtheorem*{note*}{Note}
\newtheorem*{question*}{Question}

\DeclareMathOperator{\supp}{supp}

\DeclareMathOperator{\propg}{prop}

\usepackage{enumitem}

\numberwithin{equation}{section}

\begin{document}
 
\title[ KMS states on uniform Roe algebras]{  KMS states on uniform Roe algebras}%

\date{\today} 

\author[B. M. Braga]{Bruno M. Braga}
\address[B. M. Braga]{IMPA, Estrada Dona Castorina 110, 22460-320, Rio de Janeiro, Brazil}
\email{demendoncabraga@gmail.com}
\urladdr{https://sites.google.com/site/demendoncabraga}
\thanks{B. M. Braga  was partially supported by FAPERJ (Proc. E-26/200.167/2023) and by CNPq (Proc. 303571/2022-5)}
 
\author[R. Exel]{Ruy Exel}
\address[R. Exel]{Universidade Federal de Santa Catarina, 88040-970 Florianopolis SC, Brazil}
 \email{ruyexel@gmail.com}
 \urladdr{http://www.mtm.ufsc.br/~exel/} 
 \thanks{R. Exel  was partially supported by  CNPq, Brazil}
 
\maketitle

\begin{abstract}
    We initiate the treatment of KMS states on uniform Roe algebras $\cstu(X)$ for a class of naturally occurring flows on these algebras. We show that KMS states on $\cstu(X)$ always factor through the diagonal operators $\ell_\infty(X)$. 
    We show the study of those states splits into understanding their strongly continuous KMS states and the KMS states which vanish on the ideal of compact operators. We show strongly continuous states are always unique when they exist and we give explicit formulas for them. We link the study of KMS states which vanish on the compacts to the Higson corona of $X$ and provide  lower bounds for the cardinality of the set of extreme KMS states. Lastly, we apply our theory to the $n$-branching tree: in this example, $\beta=\log(n)$ is a phase transition admitting $2^{2^{\aleph_0}}$ KMS states, no KMS states for smaller inverse temperatures, and a unique one for larger ones (the Gibbs state). Moreover, we show that the behavior of the KMS states around $\beta=\log(n)$ is chaotic.
\end{abstract}

\section{Introduction}\label{SectionIntro}

In noncommutative geometry, given a metric space $X$, one  defines  certain \cstar-algebras of operators on a Hilbert space  with the goal of   coding  certain aspects of the geometry of $X$ in \cstar-algebraic terms. When interested in the large scale geometric properties of $X$, i.e., in its coarse geometry, a well-known \cstar-algebra is    to be considered: the \emph{uniform Roe algebra of $X$}. This \cstar-algebra was introduced by Roe to study   the index theory of elliptic operators on noncompact manifolds (\cite{Roe1988,Roe1993}). The interest in   these algebras was then boosted due to  their connection with the coarse Baum-Connes   conjecture (\cite{Yu2000}). More recently, these \cstar-algebras entered the realm of  mathematical physics and researchers  interested in   topological insulators have been using them as observable algebras in order to describe    topological phases. We refer the reader to \cite{Kubota2017,EwertMeyer2019,Jones2021CommMathPhys,LudewigThiang2021CommMathPhys,Bourne2022JPhys} for the  rapidly growing  literature about   uniform Roe algebras    in mathematical physics.

The goal of this paper is to look at uniform Roe algebras under yet another point of view motivated by mathematical physics: we study KMS states on  uniform Roe algebras. Named after mathematical physicists Kubo, Martin, and Schwinger, KMS states are states   defined on any C*-algebra $A$ admitting a flow, that is, a strongly continuous one-parameter group $\{\sigma _t\}_{t\in {\mathbb R}}$ of automorphisms, thought of as the time development of observables of an idealized infinite system of particles.  Among the many equivalent definitions of such states, we adopt the one that requires our state $\varphi $ to satisfy the relation
  $$
  \varphi (ba)=\varphi \big (a\sigma _{i\beta }(b)\big ),
  $$
  for every $a$ in $A$ and every analytic element $b$ in $A$.  This condition has been noted by Kubo, Martin, and Schwinger in the late 1950's, as being satisfied by the \emph{grand canonical ensembles} in the Gibbs equilibrium formalism for finite systems.  Observing that this condition in fact characterizes the Gibbs states,
 Haag, Hugenholtz, and Winnink later proposed this as a criterion for equilibrium.

The parameter $\beta $ appearing above is the same parameter weighing the \emph {average energy} and the \emph{entropy} in the expression for the \emph{free energy} in the variational deduction of Gibbs states, and it is often thought of as the reciprocal of the \emph{temperature}.  While our abstract treatment of KMS states will not really involve the physical meaning of $\beta $, it is crucial to realize that the existence and uniqueness of KMS states depend in a very fundamental way on $\beta $, so much so that we shall refer to states satisfying the above condition as $(\sigma ,\beta )$-KMS states, following the the modern literature standards.

Crucially, among the most interesting features of KMS states is the abrupt change in behavior as $\beta $ crosses certain thresholds.  In classical infinite particle systems, a sudden change with temperature is often referred to as a \emph{phase transition}, which is what one observes when a gas liquefies when cooled down or when a magnet spontaneously looses its magnetization when heated beyond a critical temperature.  Thus, if for example there is a unique $(\sigma ,\beta )$-KMS state for every $\beta $ greater than some fixed $\beta _0$, while there are many $(\sigma ,\beta _0)$-KMS states, one says that a phase transition has happened at the critical value $\beta _0$.

Before giving a detailed description of this paper and  our main findings, we start with some basic definitions.  

\subsection{Coarse geometry and uniform Roe algebras} A map $h\colon (X,d)\to (Y,\partial)$ between metric spaces is called \emph{coarse} if for all $r>0$, there is $s>0$ such that 
\[d(x,y)<r\ \text{ implies }\ \partial(h(x),h(y))<s.\]
With coarse maps being the    morphisms of interest,  local properties of the metric spaces are irrelevant in coarse geometry and one usually restricts themselves to discrete spaces. In fact, for our goals, we will assume the metric spaces to be \emph{uniformly locally finite} (abbreviated as \emph{u.l.f.}), that is, they have the property that for each $r>0$ their balls of radius $r$ are uniformly bounded in size by a finite quantity.\footnote{A metric space with this property is also often called a \emph{metric space with bounded geometry} in the literature. Other authors call a space with bounded geometry one that is coarsely equivalent to a u.l.f.\ space. } 

Given a set $X$, $\ell_2(X)$ denotes the Hilbert space of square-summable maps $X\to \C$ and $(\delta_x)_{x\in X}$ denotes its canonical orthonormal basis. The space of bounded operators on $\ell_2(X)$ is denoted by $\cB(\ell_2(X))$ and  $\cK(\ell_2(X))$ denotes its ideal of compact operators.

\begin{definition}
    Let $(X,d)$ be a u.l.f.\ metric space. The propagation of an operator $a\in \cB(\ell_2(X))$ is defined by 
    \[\propg(a)=\sup\Big\{d(x,y)\mid a_{x,y}\coloneqq \langle a\delta_y,\delta_x\rangle\neq 0\Big\}.\]
     The $^*$-algebra of all operators with finite propagation,  denoted by $\cstu[X]$, is the \emph{algebraic uniform Roe algebra of $(X,d)$}. The norm closure of $\cstu[X]$, denoted by $\cstu(X)$, is the  \emph{uniform Roe algebra of $(X,d)$}.
\end{definition}

Uniform Roe algebras code coarse geometric properties of $X$ in terms of \cstar-algebraic properties. For instance, it is known that $X$ has Yu's property A if and only if $\cstu(X)$ is nuclear (\cite[Theorem 5.5.7]{BrownOzawa}). Also, it has been recently shown that this construction is \emph{rigid} in the sense that if the \cstar-algebras $\cstu(X)$ and $\cstu(Y)$ are isomorphic, then $X$ and $Y$ must be coarsely equivalent (\cite[Theorem 1.2]{BaudierBragaFarahKhukhroVignatiWillett2021uRaRig}).

 \subsection{Flows and KMS states on uniform Roe algebras}
Given a \cstar-algebra $A$, an action $\sigma\colon \R\curvearrowright A$ is a \emph{flow} if it is strongly continuous\footnote{ The action $\sigma$ is \emph{strongly continuous} if  $t\in \R\mapsto \sigma_t(a)\in A$ is continuous for all $a\in A$. } and  $\sigma_t\colon A\to A$ is   an isomorphism for all $t\in \R$.

  Quantum mechanical systems in thermal equilibrium can be described by their so called \emph{KMS states}.  The number $\beta$ in the definition below should be interpreted as the inverse of the temperature of the system. 

 \begin{definition}\label{DefinitionFlow}
    Let $A$ be a \cstar-algebra and $\sigma$ be a flow on $A$. For $\beta\in \R$, we say that a state $\varphi$ on $A$ is a \emph{$(\sigma,\beta)$-KMS state} if
    \[\varphi(a\sigma_{i\beta}(b))=\varphi(ba)\]
    for all   $a\in A$ and all analytic $b\in A$.\footnote{An element  $b\in A$  is \emph{analytic for $\sigma$} if the map $t\in \R\mapsto \sigma_t(b)\in A$ extends to an entire analytic map $\C\to A$.}
\end{definition}

 In order to study KMS states on uniform Roe algebras, one must first identify natural flows in them. We now introduce such flows.    Given a set $X$ and a  map $h\colon {X \to \R}$, we denote by $\bar h$ the $X$-by-$X$ diagonal matrix of reals such that its $(x,x)$-entry is $h(x)$ for all $x\in X$ and all other entries are zero. Notice that $\bar h$ canonically induces a bounded operator on $\ell_2(X)$ if and only if $h$ is bounded.

\begin{definition}
Let $X$ be a u.l.f.\ metric space and $h\colon X\to \R$ be a coarse map. We denote by $\sigma_{h}$ the flow on $\cstu(X)$ given by 
\[\sigma_{h,t}(a)=e^{it\bar h}ae^{-it\bar h}\]
for all $t\in \R $ and all $a\in \cstu(X)$.  
\end{definition}

Notice that the hypothesis on $h\colon X\to \R$ being coarse is important so that $\sigma_h$ is indeed a flow. Indeed, the action $\sigma_h$ is strongly continuous if and only if $h$ is coarse  (see Proposition \ref{PropItIsAFlow}). All flows on uniform Roe algebras considered in this paper will be of the form above for some appropriate $h\colon X\to \R$. In order to have any hope of understanding the KMS states for those flows, we must first understand the analytic elements of $\cstu(X)$ or, more precisely, a $^*$-subalgebra of analytic operators of $\cstu(X)$ which is dense in it. We have:
 
\begin{proposition}
Let $X$ be a u.l.f.\ metric space and $h\colon X\to \R$ be a  map.

\begin{enumerate}
    \item\label{PropAnalyticElements.Item1} If $h$ is bounded, then every element of $\cstu(X)$ is analytic for $\sigma_h$
    \item\label{PropAnalyticElements.Item2} If $h$ is coarse, then every element of $\cstu[X]$ is analytic for $\sigma_h$. 
\end{enumerate} \label{PropAnalyticElementsINTRO}
\end{proposition}

The reader may wonder   how strong   is the restriction of only working with flows of the form above. As we show in Proposition \ref{PropFlowsPreservinglinfty}, if   $\sigma\colon \R\curvearrowright \cstu(X) $ is an arbitrary flow which leaves the Cartan masa $\ell_\infty(X)$ invariant, i.e., $\sigma_t(\ell_\infty(X))\subseteq \ell_\infty(X)$  for all $t\in \R$, then there is a coarse map $h\colon X\to \R$ such that $\sigma=\sigma_h$.\footnote{ We thank Stuart White for raising the possibility that this could be true.} This corroborates to 
our claim that such flows form a very natural  and general class of flows on those algebras.

\subsection{Main results}\label{SubsectionIntroMain}
It is often common in the study of KMS states on a given \cstar-algebra $A$ that there is some ``natural'' \cstar-subalgebra $B\subseteq A$ and a conditional expectation $E\colon A\to B$ such that the KMS states $\varphi\colon A\to \C$ factor through $E$. We show that this is also the case in our setting with the ``natural'' \cstar-subalgebra through which the KMS states factor being the \cstar-algebra of all bounded maps $X\to \C$, denoted by $\ell_\infty(X)$. Precisely, throughout these notes, we identify  $\ell_\infty(X)$  with the \cstar-algebra of diagonal operators on $\ell_2(X)$  in the usual way:  given $a=(a_x)_{x\in X}\in \ell_\infty(X)$ and $\xi=(\xi_x)_{x\in X}\in \ell_2(X)$, we let \[a\xi=(a_x\xi_x)_{x\in X}\in \ell_2(X).\]
Given $A\subseteq X$, $\chi_A\in \ell_\infty(X)$ denotes the canonical orthogonal projection  $\ell_2(X)\to\ell_2(A)$.

We show the following: 

\begin{theorem} 
Let $X$ be a u.l.f.\ metric space,  $h\colon X\to \R$ be a  coarse map, and $\beta\in \R$. If $\varphi$ is a $(\sigma_{h},\beta)$-KMS state on $\cstu(X)$, then $\varphi=\varphi\circ E$, where $E\colon \cstu(X)\to \ell_\infty(X)$ is the canonical conditional expectation (see Figure \ref{Fig1}).
\begin{figure*}[h] \centerline{
\xymatrix{ 
        \cstu(X)\ar[dr]_E  \ar[rr]^\varphi& &\C  \\
         & \ell_\infty(X) \ar[ur]_{\varphi\restriction \ell_\infty(X)}&  }}\caption{KMS states on $\cstu(X)$ factor through $\ell_\infty(X)$, see Subsection \ref{SubsectionFactorlInfty} for the precise definition of $E$.}\label{Fig1}
\end{figure*}
\label{ThmFactorsCondExpINTRO} 
\end{theorem}

Theorem \ref{ThmFactorsCondExpINTRO}  is an extremely powerful tool in our study of KMS states on uniform Roe algebras and most of our results deeply depend on it. For instance, it allows us to understand the case of a flow given by a bounded map $h\colon X\to \R$ in terms of amenability: for $h$ bounded,  $\cstu(X)$ has a $(\sigma_h,\beta)$-KMS states if and only if $X$ is amenable (see Theorem \ref{ThmBoundedhAmenable}). Moreover, Theorem \ref{ThmFactorsCondExpINTRO}  allows us to reduce the study of KMS states on uniform Roe algebras to two cases (see Proposition \ref{PropKMSStateVanishComp}): 

\begin{enumerate}[label=(\Roman*)]
\item strongly continuous KMS states, and \item KMS states which vanish on the  the ideal of compact operators.
\end{enumerate}
 The strongly continuous case is the simplest one and the next result summarizes  what happens:
 
 \begin{theorem}
     Let $X$ be a u.l.f.\ metric space,  $h\colon X\to \R$ be a coarse map, and $\beta\geq0$. There are strongly continuous $(\sigma_{h},\beta)$-KMS states on $\cstu(X)$ if and only if 
    \[ Z(\beta) := \text{tr}(e^{-\beta \bar h}) = \sum_{x\in X}e^{-\beta h(x)}<\infty.\]
Moreover,      a function  $\varphi\colon \cstu(X)\to \C$  is a strongly continuous $(\sigma_{ h},\beta)$-KMS state on $\cstu(X)$ if and only if
\begin{equation} 
\varphi(a)=
\frac{\text{tr}(e^{-\beta \bar h}a)}{\text{tr}(e^{-\beta \bar h})} =
\frac{1}{Z(\beta)}\sum_{x\in X}a_{x,x}e^{-\beta h(x)}
\end{equation}
     for all $a=[a_{x,y}]\in \cstu(X)$. In particular, whenever they exist, strongly continuous $(\sigma_h,\beta)$-KMS states are unique.\label{ThmStrongContKMSINTRO}
\end{theorem}

In other words, the strongly continuous KMS states are exactly the Gibbs states 
provided $e^{-\beta\bar h}$ is trace class (see \cite[Section 6.2.2]{BratteliRobinsonBookII1997}).  This is of course no big surprise since the  strongly continuous states on any
operator algebra containing the compacts correspond precisely  with the strongly continuous states defined
on the whole $\cB(\ell_2(X))$.

With the strongly continuous case being well understood, we then proceed to study the much more interesting case of KMS states which vanish on the compact operators. This property allows us to factor those states through the \emph{uniform Roe corona of $X$}. 

\begin{definition} (\cite[Definition 1.2]{BragaFarahVignati2018AdvMath}).
Let $X$ be a u.l.f.\ metric space. The \emph{uniform Roe corona of $X$} is the \cstar-algebra given by  
\[\roeq(X)=\cstu(X)/\cK(\ell_2(X)).\]
We denote by $\pi=\pi_X\colon \cstu(X)\to \roeq(X)$ the canonical quotient map.\label{DefintionRoeq} 
\end{definition}

A state  $\varphi$  on $\cstu(X)$ which vanishes on $\cK(\ell_2(X))$ gives rise to a well-defined  state $\psi$ on $\roeq(X)$ determined by 
\[\psi(\pi(a))=\varphi(a), \text{ for all }\ a\in \cstu(X).\]
Moreover, given a coarse map   $h\colon X\to \R$, the flow $\sigma_h$ canonically induces a flow on the corona $\roeq(X)$.  Precisely, as  $\sigma_{h}$ leaves $\cK(\ell_2(X))$ invariant, i.e., 
\[\sigma_{h,t}(\cK(\ell_2(X)))\subseteq \cK(\ell_2(X))\ \text{ for all }\ t\in \R,\]
we obtain a flow $\sigma_{h}^\infty$ on $\roeq(X)$ by letting
\[\sigma_{h,t}^\infty(\pi(a))=\pi(\sigma_{h,t}(a))\ \text{ for all }\ a\in \cstu(X)\ \text{ and all }\ t\in \R.\] 
In other words, $\sigma_{h}^\infty$ is a flow on $\roeq(X)$ which makes the following diagram commute. 
\begin{equation*}
    \xymatrix{ 
        \cstu(X) \ar[d]_\pi  \ar[r]^{\sigma_{h,t}} &\cstu(X)  \ar[d]^\pi  \\
\roeq(X)\ar[r]_{\sigma_{h,t}^\infty} &\roeq(X)  }
\end{equation*}

We show that the study of $(\sigma_h,\beta)$-KMS states on $\cstu(X)$ which vanish on the ideal of compact operators completely reduces to the study of $(\sigma_h^\infty,\beta)$-KMS states on $\roeq(X)$ in a canonical way. Precisely:

\begin{proposition}  
Let $X$ be a u.l.f.\ metric space, $h\colon X\to \R$ be a coarse map, and $\beta\in \R$. A state $\psi$ on $\roeq(X)$ is a $(\sigma_{h}^\infty,\beta)$-KMS state   if and only if $\varphi=\psi\circ \pi$ is a $(\sigma_{h},\beta)$-KMS state on $\cstu(X)$. Moreover, the assignment 
\[\psi\mapsto \varphi=\psi\circ \pi\]
is an affine isomorphism from the set of all $(\sigma_h^\infty,\beta)$-KMS states on $\roeq(X)$ to the set of all  $(\sigma_h,\beta)$-KMS states on $\cstu(X)$ which vanish on $\cK(\ell_2(X))$. \label{PropRoeqQStatesAffineIsoINTRO}
\end{proposition}

Guided by Proposition \ref{PropRoeqQStatesAffineIsoINTRO}, we then focus on KMS states on the corona algebra $\roeq(X)$. For that, we show some general results about KMS states on arbitrary \cstar-algebras with respect to arbitrary flows (see Section \ref{SectionIntermission} for details). In a nutshell, we show that the 
  extreme KMS states on an arbitrary \cstar-algebra $A$ are influenced by 
  %
  the center of $A$, denoted by $\cZ(A)$, and its \cstar-subalgebras.  Returning to our coarse setting, this brings up a seemingly unexpected link between KMS states on uniform Roe algebras and the \emph{Higson corona} of metric spaces. More precisely, given a u.l.f.\ metric space $X$, we denote its \emph{Higson compactification} by $hX$ and its \emph{Higson corona} by $\nu X=h X\setminus X$.\footnote{For brevity, we refer the reader to Definition \ref{DefinitionHigsonCor} for the precise definition of the Higson compactification/corona.} The space of continuous functions on the Higson compactification, $C(hX)$, is canonically seen as a \cstar-subalgebra of $\ell_\infty(X)$, which in turn allow us to canonically identify the continuous functions on its corona, $C(\nu X)$, with a \cstar-subalgebra of $\roeq(X)$. Under this identifications, it has been recently shown that   \[\cZ(\roeq(X))=C(\nu X)\]
(see \cite[Proposition 3.6]{BaudierBragaFarahVignatiWillett2023}).

This link between KMS states and  the Higson corona is essential in the analysis of KMS states which vanish on the compacts.  Precisely,  the next result summarizes our findings on the topic.  

\begin{theorem}
 Let $X$ be a u.l.f.\ metric space, $h\colon X\to \R$ be a coarse map, and $\beta\in \R$.

 \begin{enumerate}
 \item For any  extreme $ (\sigma_{h}^\infty,\beta)$-KMS state   $\psi$ on $\roeq(X)$, there is $x\in \nu X$ such that \[\psi(a)=a(x) \ \text{ for all }\ a\in C(\nu X).\]
     \item Suppose there is a $(\sigma_{h}^\infty,\beta)$-KMS state on $\roeq(X)$  whose restriction to $C(\nu X)$ is faithful. Then, for any $x\in \nu X$, there is an extreme $(\sigma_{h}^\infty,\beta)$-KMS state $\psi$ on $\roeq(X)$  such that \[\psi(a)=a(x) \ \text{ for all }\ a\in C(\nu X).\]
 \end{enumerate}
 \label{ThmExtRoeqINTRO}
\end{theorem}

 Our methods give us a strong control on the support of KMS states on $\cstu(X)$. In order to state this control,   a definition is in place.  

 \begin{definition}\label{DefinitionSupp}
Let $X$ be a u.l.f.\ metric space, $x\in \nu X$, and   $\varphi$ be a state on $\cstu(X)$. We say that $\varphi$ is \emph{supported on $x$} if for all neighborhoods $U\subseteq h X$ of $x$, we have $\varphi(\chi_{U\cap X})=1$.
 \end{definition}

 \begin{theorem}\label{ThmUltimoINTRO}
 Let $X$ be a u.l.f.\ metric space,    $h\colon X\to \R$ be a coarse map, and $\beta\in \R$.  The following holds:
 \begin{enumerate}
 \item Any extreme  $(\sigma_h,\beta)$-KMS state on $\cstu(X)$ which vanishes on the compacts is supported at some element of $\nu X$.
 \item If there is a  $(\sigma_h,\beta)$-KMS state on $\cstu(X)$ which vanishes on the compacts and such that its induced state on $\roeq(X)$ is faithful on $C(\nu X)$, then for every $x\in \nu X$ there is a   $(\sigma_h,\beta)$-KMS state on $\cstu(X)$ supported on $x$.
 \end{enumerate}
 \end{theorem}

 In fact,  both  Theorems \ref{ThmExtRoeqINTRO}  and \ref{ThmUltimoINTRO}
have  versions that hold with  $C(\nu X)$ being substituted by   arbitrary unital \cstar-subalgebras of $C(\nu X)$ (see Theorems  \ref{ThmExtRoeqGENERAL} and   \ref{ThmUltimoGENERAL}).

In Subsection \ref{SubsectionSize}, we show that the Higson corona of any infinite u.l.f.\  metric space contains $2^{2^{\aleph_0}}$ elements (see Theorem \ref{ThmSizeCorona}). This result has been first obtained in \cite[Theorem 3]{Keesling1994TopProc}, but we chose to present an alternative and self-contained proof here for the readers convenience.  As a consequence of this result,    Theorem \ref{ThmExtRoeqINTRO} and Proposition \ref{PropRoeqQStatesAffineIsoINTRO} above imply that  if there is a $(\sigma_{h}^\infty,\beta)$-KMS state on $\roeq(X)$  whose restriction to $C(\nu X)$ is faithful, then there are  $2^{2^{\aleph_0}}$ extreme KMS states in both $\roeq(X)$ and $\cstu(X)$ (see Corollary \ref{CorSize}).

\subsection{Applications}\label{SubsectionIntroAppli}
Our methods can be applied to specific metric spaces. Notice that Theorem \ref{ThmStrongContKMSINTRO} implies that if the balls of $X$ have polynomial growth, then $\cstu(X)$ will have $(\sigma_{h},\beta)$-KMS states for any $\beta>0$ and any  ``reasonable''  $h\colon X\to \R$. Indeed, suppose $h$ is such that there is $L>0$ and $x_0\in X$ for which 
\[h(x)\geq \frac{d(x,x_0)}{L}-L \text{ for all }\ x\in X.\]
Suppose now $p$ is a polynomial controling the growth of the balls of $X$, i.e.,  every  ball in $X$ centered at $x_0$ of radius $r$ has at most $p(r)$ elements. Then, 
  the series $\sum_{x\in X}e^{-\beta h(x)}$ converges to a finite number for any $\beta>0$. Therefore, in order to find examples with interesting phase transition, it is advisable to look for metric spaces with large growth. This makes the $n$-branching trees  natural spaces to apply our theory to.

  We point out   that, due to the technical aspects of Theorems \ref{ThmExtRoeqINTRO} and \ref{ThmUltimoINTRO}, the result below is not a mere corollary of the results above and a deeper analysis of   Higson coronas as well as of the weak$^*$-limit of their strongly continuous KMS states is needed. The study of invariant means on semigroups developed by Chou in \cite{Chou1969PAMS} is also essential for the precise computation of the cadinality of extreme KMS states presented below.

   Given $n\in\N$, let $T_n$ denote the $n$-branching tree, i.e.,  $T_n=\{\emptyset\}\cup\bigcup_{k=1}^\infty\{1,\ldots, n\}^k$ and we endow $T_n$ with its canonical graph distance (see Section \ref{SectionApli} for details).  The \emph{branches} of $T_n$ are denoted by $[T_n]$, i.e., $[T_n]=\{1,\ldots, n\}^\N$. Given $\bar x=(x_i)_{i=1}^\infty\in [T_n]$, we let $\bar x | k=(x_1,\ldots, x_k)\in T_n$ and $\bar x| k^\smallfrown T_n$ denotes the words in $T_n$ which start with $\bar x| k$. 

\begin{theorem} 
Given $n\in\N$, let $T_n$ denote the $n$-branching tree endowed with its graph distance $d$ and   let $\emptyset$ denote its root. Let  $h\colon T_n\to \R$ be given by $h(x)=d(x,\emptyset)$ for all $x\in T_n$. Then there is a  $(\sigma_{h},\beta)$-KMS state on $\cstu(T_n)$ if and only if   $\beta\geq \log(n)$. Moreover,

\begin{enumerate}
    \item \label{ThmnBranchingTreeItem2}
For $\beta> \log(n)$,   there is a unique $(\sigma_{h},\beta)$-KMS state on $\cstu(T_n)$  and   this state is strongly continuous.

\item \label{ThmnBranchingTreeItem3} For $\beta=\log(n)$, the $(\sigma_{h},\beta)$-KMS states on $\cstu(T_n)$ all vanish on $\cK(\ell_2(T_n))$. Moreover,  for all $\bar x\in [T_n]$, there are  $2^{2^{\aleph_0}}$  extreme   $(\sigma_{h},\beta)$-KMS states $\varphi$   on $\cstu(T_n)$ such that  
\[\varphi(\chi_{\bar x|k^\smallfrown T_n})=1 \ \text{ for all }\  k\in\N.\]
Conversely, any extreme $(\sigma_{h},\beta)$-KMS state   on $\cstu(T_n)$ satisfies the above for an appropriate $\bar x\in [T_n]$.
\end{enumerate}\label{ThmnBranchingTreeCOMPLETEINTRO}
\end{theorem}

For inverse temperature $\beta>\log(n)$, we actually have a precise formula for its unique KMS state (see Theorem \ref{ThmnBranchingTree}).

Finally, in Section \ref{SectionChaotic}, we discuss a somewhat unusual phenomenon known as \emph{chaotic
convergence} of KMS states.  In order to explain what this means, consider a flow $\sigma $ on a C*-algebra $A$ admitting a unique KMS state at inverse
temperature $\beta $, say $\varphi _\beta $, for every $\beta $ in an interval of the form $(\beta _0,\beta _0+\varepsilon )$, so that it makes sense to ask whether
or not the limit
  $$
  \lim_{\beta \to \beta _0^+} \varphi _\beta 
  $$
  exists (here the limit should be taken with  respect to the weak$^*$ topology).
  The most commonly observed behavior
  (see \cite{vanEnterRusze2007JStatPhy,ChazottesHochman2010CommMathPhy,CoronelRivera-Letelier2015JStaPhy,BissacotGaribaldiThieullen2018ErgDymSym})
  is when this limit exists, even when $\beta _0$ is critical, that is,  even when there are multiple $(\sigma ,\beta_0 )$-KMS states.

By \emph{chaotic convergence} of KMS states it is meant a situation where the above fails in the sense that there are different
sequences $\beta _n$ converging to $\beta _0$ from above for which the corresponding limit states differ.  This chaotic behavior
has been observed for ground states \cite{BissacotGaribaldiThieullen2018ErgDymSym}, that is, regarding the limit as $\beta  \to 
\infty $, but we are not aware of too many situations where this phenomenon happens at finite temperatures.

As detailed in
Theorem \ref{ThmRodrigoBissacot} below, we analyze this question for  $\cstu(T_n)$ as $\beta$ approaches $\log(n)$ from
above, showing that such chaotic behavior is indeed present.

\color{black}

\section{Basics on KMS states on uniform Roe algebras} 

In this section, we start our study of KMS states of uniform Roe algebras and prove several general properties which will be essential throughout these notes. We also present some simple examples by studying the KMS states of the simplest coarse space: $\{n^2\mid n\in\N\}$. We start this section introducing some   notation which was left out from Section \ref{SectionIntro}. 

Given a set $X$ and   $x,y\in X$, we let $e_{x,y}\in \cB(\ell_2(X))$ be the rank 1 partial isometry  sending  $\delta_y$ to $\delta_x$. If   $A\subseteq X$, we let 
\[\chi_A=\SOTh\sum_{x\in X}e_{x,x};\]
where the letters $\mathrm{SOT}$ above mean that the sum converges with respect to the strong operator topology.  In other words, $\chi_A$ is the canonical orthogonal projection $\ell_2(X)\to \ell_2(A)$. Under the identification of $\ell_\infty(X)$ with the \cstar-subalgebra of $\cstu(X)$ consisting of the  diagonal operators, we have that $\chi_A\in \ell_\infty(X)$ for all $A\subseteq X$. The \cstar-algebra of functions $X\to \C$ which vanish at infinity is identified with the compact operators in $\ell_\infty(X)$, i.e., 
\[c_0(X)=\ell_\infty(X)\cap \cK(\ell_2(X)).\] 

The following description of  operators in  $\cstu[X]$   will be very useful for our goals: Firstly, recall that a \emph{partial bijection of $X$} is a bijection $f\colon A\to B$ between subsets $A$ and $B$ of $X$. If moreover \[\sup_{x\in A}d(x,f(x))<\infty,\] then we say that $f$ is a \emph{partial translation}. Given any partial translation $f\colon A\subseteq X\to B\subseteq X$, we define an operator $v_f$ on $\ell_2(X)$ by letting 
    \begin{equation}\label{EqDefvf}
    v_f\delta_x=\left\{\begin{array}{ll}
       \delta_{f(x)},  & x\in A , \\
        0, & x\notin A. 
    \end{array}\right.\end{equation}
So, each $v_f$ is a partial isometry and the algebraic uniform Roe algebra is linearly spanned by products of elements in $\ell_\infty(X)$ by those partial isometries. Precisely, we have  
    \[\cstu[X]=\mathrm{span}\Big\{av_f\mid a\in \ell_\infty(X) \text{ and } f\text{ is a partial translation on }X\Big\}\] 
(see  \cite[Lemma 2.4]{SpakulaWillett2017}
    for details).

\subsection{Flows and analytic elements} Our very first result shows that the actions $\sigma_h\colon \R\curvearrowright \cstu(X)$ are indeed   flows if and only if $h$ is coarse. 

\begin{proposition}\label{PropItIsAFlow}
    Let $X$ be a u.l.f.\ metric space and $h\colon X\to \R$ be a pode  map. Then $h$ is coarse if and only if the action $\sigma_h$ given by Definition \ref{DefinitionFlow} is strongly continuous, i.e., \begin{equation}\label{EqPropItIsAFlow}
    t\in \R\mapsto \sigma_{h,t}(a)\in \cstu(X)
    \end{equation} is continuous for all $a\in \cstu(X)$.
\end{proposition}

\begin{proof}
 Suppose first that $h$ is coarse.  Since $\cstu[X]$ is dense in $\cstu(X)$, it is enough to show that  the map in \eqref{EqPropItIsAFlow} is continuous for each $a\in \cstu[X]$. Moreover, since $\cstu[X]$ is spanned by the subset of all $av_f$, for $a\in \ell_\infty(X)$ and $f\colon A\subseteq X\to B\subseteq X$ a partial translation, it is enough to notice that \eqref{EqPropItIsAFlow} holds for all such elements $av_f$. Fix such $a$ and $f$. Then, as $f$ is a partial bijection, we have that
   \begin{align}\label{EqPropItIsAFlow2}\|\sigma_{h,t}(av_f)-\sigma_{h,s}(av_f)\|&=\|e^{it\bar h}av_fe^{-it\bar h}-e^{is\bar h}av_fe^{-is\bar h}\|\\
   &=\sup_{x\in A}\Big(e^{it (h(f(x))-h(x))}-e^{is (h(f(x))-h(x))}\Big)a_{f(x),f(x)}.\notag
   \end{align}
Since $f$ is a partial translation and $h$ is coarse, we have
\[\sup_{x\in A}|h(f(x))-h(x)|<\infty.\]
Therefore, it follows from \eqref{EqPropItIsAFlow2} and the intermediate value theorem that  \[t\in \R\mapsto \sigma_t(av_f)\in \cstu(X)\]
is continuous.

Suppose now that the action $\sigma_h$ is strongly continuous. Suppose towards a contradiction that $h$ is not coarse. Then there is $r>0$, and sequences $(x_i)_i$ and $(y_i)_i$ in $X$ such that $\lim_i|h(x_i)-h(y_i)|=\infty$ and $d(x_i,y_i)\leq r$ for all $i\in\N$. As $X$ is u.l.f., those sequences cannot be bounded, so, by going to a subsequence if necessary, we assume that $(x_i)_i$ and $(y_i)_i$ are sequences of distinct points of $X$. We can then define a map 
\begin{align*}f\colon \{x_i\mid i\in \N\}&\to \{y_i\mid i\in \N\}\\
x_i&\mapsto y_i
\end{align*}
and this map is a   partial translation. So, $v_f\in \cstu[X]$ and, since $\sigma_h$ is strongly continuous, we have that 
\[\lim_{t\to 0}\|\sigma_t(v_f)-v_f\|=0.\]
Fix $\delta>0$ such that 
\[|t|<\delta\ \text{ implies } \ \|\sigma_t(v_f)-v_f\|<2.\]

Notice now that
\begin{align*}
    \|\sigma_t(v_f)-v_f\|&=
    \|e^{it\bar h}v_fe^{-it\bar h}-v_f\|\\
    &=\sup_{x\in X}\Big|e^{it (h(f(x))- h(x))}-1\Big|
    \\
    &\geq  \sup_{i\in \N}\Big|e^{it (h(y_i)- h(x_i))}-1\Big|.
\end{align*}
Hence, picking $i\in\N$ large enough so that 
\[t=\frac{\pi}{|h(y_i)-h(x_i)|}<\delta,\]
we obtain that $\|\sigma_t(v_f)-v_f\|\geq 2$; contradiction.
\end{proof}

We now show that our choice of only dealing with flows of the form $\sigma_h$ for some coarse map $h\colon X\to \R$ does not represent a big restriction in a sense.

\begin{proposition}\label{PropFlowsPreservinglinfty}
    Let $X$ be a u.l.f.\ metric space and let $\sigma\colon \R\curvearrowright \cstu(X)$ be a flow leaving $\ell_\infty(X)$ invariant, i.e., $\sigma(\ell_\infty(X))\subseteq \ell_\infty(X)$ for all $t\in \R$. Then, there is a coarse map $h\colon X\to \R$ such that $\sigma=\sigma_h$.
\end{proposition}

\begin{proof}
    We first notice that the condition  of $\sigma\colon \R\curvearrowright \cstu(X)$ leaving $\ell_\infty(X)$ invariant implies that $\sigma_t$ is the identity on $\ell_\infty(X)$ for all $t\in \R$. Indeed, as $\sigma_0$ is by hypothesis the identity on $\cstu(X)$, we have that $\sigma_0(e_{x,x})=e_{x,x}$ for all $x\in X$. As $\sigma_t$ is an isomorphism for all $t\in \R$, $\sigma_t(e_{x,x})$ must be a projection for all $t\in \R$ and all $x\in X$. Therefore, since $t\in \R\mapsto \sigma_t(e_{x,x})\in \ell_\infty(X)$ is continuous, this shows that $\sigma_t(e_{x,x})=e_{x,x}$ for all $t\in \R$ and all $x\in X$. Hence, $\sigma_t$ must be the identity on $c_0(X)$ for all $t\in \R$. As isomorphisms of uniform Roe algebras are strongly continuous (\cite[Lemma 3.1]{SpakulaWillett2013}), this shows that each $\sigma_t$ is the identity on $\ell_\infty(X)$ are desired.
 
    Fix $x\in X$.  For each $\xi\in \ell_2(X)$, let $r_\xi$  be the rank one operator given by
    \[r_\xi\zeta=\langle \zeta,\delta_x\rangle\xi \ \text{ for all }\ \zeta\in \ell_2(X).\]
 For each $t\in \R$, define an operator $u_t$ on $\ell_2(X)$ by letting 
    \[u_t\xi=\sigma_t(r_\xi)\delta_x\ \text{ for all }\ \xi\in \ell_2(X).\]

\begin{claim}\label{ClaimFlow1}
  We have   \[\sigma_t(a)=u_tau_{-t}\ \text{ for all} \ a\in \cstu(X)\ \text{ and all }\ t\in \R.\]  In particular, $u_t\in \ell_\infty(X)$ for all $t\in\R$.
\end{claim}

\begin{proof}
     First notice that \begin{equation}\label{EqFLowForm1}
     ae_{x,x}=r_{a\delta_x}\ \text{ for all }\ t\in \R\ \text{ and all }a\in \cstu(X). 
     \end{equation}
     Hence, by the arbitrariness of $a$ above, this implies that 
    \[u_tau_{-t}\xi=u_ta\sigma_{-t}(r_\xi)\delta_x=\sigma_t(a\sigma_{-t}(r_\xi))\delta_x=\sigma_t(a)r_\xi\delta_x=\sigma_t(a)\xi\]
    for all $\xi\in \ell_2(X)$, all $t\in \R$, and all $a\in \cstu(X)$.
    
For the last claim, notice that, as each $\sigma_t$ is the identity on $\ell_\infty(X)$, the previous paragraph  implies that each $u_t$ commutes with the elements of $\ell_\infty(X)$. As $\ell_\infty(X)$ is a maximal abelian subalgebra of $\cstu(X)$, this gives that $u_t\in \ell_\infty(X)$ for all $t\in \R$.
\end{proof}

\begin{claim}\label{ClaimFlow2}
The family $(u_t)_t$ is a one-parameter unitary group, i.e.,  $t\in\R\mapsto u_t\xi\in\ell_2(X)$ is continuous for all $\xi\in \ell_2(X)$,  $u_{t+s}=u_tu_s$ for all $t,s\in \R$, and each $u_t$ is a unitary,
\end{claim}

\begin{proof}
    First notice that, as $t\in \R\mapsto\sigma_t(r_\xi)\in \cstu(X)$ is continuous,     $t\in \R\mapsto u_t\xi\in \ell_2(X)$ is also continuous   for all $\xi\in \ell_2(X)$. Also, using  \eqref{EqFLowForm1}, we have
\[ u_t(u_s\xi)=u_t(\sigma_s(r_\xi)\delta_x)=\sigma_t(\sigma_s(r_\xi))\delta_x=\sigma_{t+s}(r_\xi)\delta_x=u_{t+s}\xi\]
    for all $\xi\in \ell_2(X)$ and all $t,s\in \R$. Finally, as each $u_t$ is an element of $\ell_\infty(X)$ with   norm at most one satisfying $u_{t}u_{-t}=1$, this also shows that $u_t$ is a unitary.
\end{proof}
    
By Claims \ref{ClaimFlow1}  and \ref{ClaimFlow2}, there is  map $h\colon X\to \R$ such that 
\[u_t=e^{it\bar h}\ \text{ for all }\ t\in \R.\]
Therefore, by Claim \ref{ClaimFlow1}, we have that $\sigma=\sigma_h$. By Proposition \ref{PropItIsAFlow}, it follows that $h$ must be coarse.
    \end{proof}

In order to study the KMS states on uniform Roe algebras which are given by the flows defined above, it is essential to understand the analytic elements of this flow. This is precisely the content of Proposition \ref{PropAnalyticElementsINTRO}.

\begin{proof}[Proof of Proposition \ref{PropAnalyticElementsINTRO}]
\eqref{PropAnalyticElements.Item1} If $h$ is bounded, $\bar h$ is a bounded operator on $\ell_2(X)$. Therefore, the analyticity of $e^z$ gives that 
\[z\in \C\to e^{-iz\bar h}ae^{iz\bar h}\in \cstu(X)\]
is analytic for all $a\in \cstu(X)$.

\eqref{PropAnalyticElements.Item2} 
Since $\cstu[X]$ is spanned by the subset of all $av_f$, for $a\in \ell_\infty(X)$ and $f\colon A\subseteq X\to B\subseteq X$ a partial translation, it is enough to show that each such $av_f$ is analytic. Fix such $a$ and $f\colon A\subseteq X\to B\subseteq X$, and  let
$g\colon X\to \R$ be given by  
\[g(x)=\left\{\begin{array}{ll}
h(f(x))-h(x),& x\in A,\\
0,& x\not\in A.
\end{array}\right.\]
A simple computation   gives that
\[\sigma_{h,t}(av_f)=e^{it\bar g}av_f\]
for all $t\in \R$. As $d(f(x),x)\leq r$ for all $x\in A$,  $g$ is   bounded. Then, the analyticity of $e^z$ implies the that
\[z\in \C\to e^{iz\bar g}av_f\in \cstu(X)\] is analytic; so, $av_f$ is analytic.
\end{proof}

\subsection{Factoring KMS-states through $\ell_\infty(X)$}\label{SubsectionFactorlInfty} It is often common in the study of KMS states on a given \cstar-algebra $A$ that there is some ``natural'' \cstar-subalgebra $B\subseteq A$ and a conditional expectation $E\colon A\to B$ such that the KMS states $\varphi\colon A\to \C$ factor through $E$; precisely, $\varphi=\varphi \circ E$, so the diagram below commutes.
\begin{equation*}
    \xymatrix{ 
        A \ar[dr]_E  \ar[rr]^\varphi& &\C  \\
         & B \ar[ur]_{\varphi\restriction B}&  }
\end{equation*} 
We now show that this also happens  with KMS state on uniform Roe algebras. 

Recall, $\ell_\infty(X)$ is a Cartan masa of $\cstu(X)$ and the conditional expectation $E\colon \cstu(X)\to \ell_\infty(X)$ is simply deleting the matrix entries of the operators on $\cstu(X)$ which are not in the main diagonal.  Precisely, the canonical conditional expectation $E\colon \cstu(X)\to \ell_\infty(X)$ is defined as follows:
\[\langle E(a)\delta_x,\delta_y\rangle= \left\{\begin{array}{ll} a_{x,x},& x=y,\\
0,& x\neq y,\end{array}\right.\]
for all $a=[a_{x,y}]\in \cstu(X)$ and all $x,y\in X$.

\begin{proof}[Proof of Theorem \ref{ThmFactorsCondExpINTRO}]
    As $\cstu[X]$ is dense in $\cstu(X)$, it is enough to show that $\varphi(a)=\varphi(E(a))$ for all $a\in \cstu[X]$. Moreover, as   $\cstu[X]$ is the span of all $av_f$, where $ a\in \ell_\infty(X)$ and  $f$ is a partial translation on $X$,   it is enough   to show that $\varphi(av_f)=0$ for all $a\in \ell_\infty(X)$ and all partial translations $f\colon A\subseteq X\to B\subseteq X$ such that $f(x)\neq x$ for all $x\in A$; fix $a$ and $f$ as such. 
    
    Let $r=\sup_{x\in A}d(x,f(x))$;  as $f$ is a partial translation, $r$ is finite. As $X$ is u.l.f., there is a partition \[A=A_1\sqcup\ldots\sqcup A_n\] such that each $A_i$ is $2r$-separated, i.e.,  $d(x,y)>2r$ for all $i\in \{1,\ldots, n\}$ and all distinct $x,y\in A_i$. Therefore, 
    \[d(x,f(y))\geq d(x,y)-d(y,f(y))>r\]
    for all $i\in \{1,\ldots, n\}$ and all distinct $x,y\in A_i$; in particular, $x\neq f(y)$. Moreover, as   $f(x)\neq x$ for all $x\in A$, this shows that 
    \begin{equation}\label{Eq.IntAfA}
    A_i\cap f(A_i)=\emptyset
    \end{equation} for all $i\in \{1,\ldots, n\}$. 
    
    For each $i\in \{1,\ldots, n\}$, let $f_i=f\restriction A_i$. So, \eqref{Eq.IntAfA} implies that $\chi_{A_i}v_{f_i}=0$ for all $i\in \{1,\ldots, n\}$. Therefore, since  \[\chi_{A_i}\sigma_{h,i\beta}(av_{f_i})=\chi_{A_i}e^{-\beta \bar h}v_{f_i}e^{\beta\bar h}=e^{-\beta \bar h}\chi_{A_i}v_{f_i}e^{\beta\bar h}=0,\] we conclude that 
    \[\varphi(av_{f_i})=\varphi( av_{f_i}\chi_{A_i})=\varphi(\chi_{A_i}\sigma_{h,i\beta}(av_{f_i}))=0.\]
    Since $v_f=v_{f_1}+\ldots+v_{f_n}$, this finishes the proof.
\end{proof}

As KMS states on uniform Roe algebras factor through the canonical conditional expectation $E\colon \cstu(X)\to \ell_\infty(X)$, it will be very useful to have a condition on when a state $\varphi$ on $\cstu(X)$ satisfies the KMS condition which depends only on operators on $\ell_\infty(X)$. We first introduce some notation which will be used in the next proof. Given   $a=(a_y)_y\in \ell_{\infty}(X)$ and a partial bijection $f\colon A\subseteq X\to B\subseteq X$, we let $a_{\circ f}\in \ell_\infty(X)$ be the operator given by 
\[a_{\circ f}\delta_x=\left\{\begin{array}{ll}
   a_{f(x)}\delta_x,  &x\in A,  \\
    0, &  x\not\in A,
\end{array}\right.\]
for all $x\in X$.\footnote{Here is a justification for this cumbersome notation: if $a\in \ell_\infty(X)$, then one can see $a$ as a bounded sequence, say $a=(a_x)_{x\in X}$. Then $a_{\circ f}$ is   the extension of $(a_{f(x)})_{x\in A}$ to the whole $X$ by letting the coordinates not in $ A$ be zero.} 

\begin{theorem}\label{ThmSufficientKMS}
    Let $X$ be a u.l.f.\ metric space,  $h\colon X\to \R$ be coarse, and $\beta\in \R$. Suppose $\varphi$ is a state on $\ell_\infty(X)$. Then $\varphi$ satisfies  
     \begin{equation}\label{EqVarphiSufficient}
        \varphi(\chi_{f(A)} )
=\varphi\Big(\chi_A e^{\beta (\overline{  h- h\circ f})} \Big)
\end{equation}
  for all   partial translations $f\colon A\to f(A)$ on $X$    if and only if   $\varphi\circ E$ is a  $(\sigma_{h},\beta)$-KMS state on $\cstu(X)$; where $E\colon \cstu(X)\to \ell_\infty(X)$ is the canonical conditional expectation. 
\end{theorem}

\begin{proof}
Suppose first that  $\varphi$ is a $(\sigma_{h},\beta)$-KMS on $\cstu(X)$. Let $f\colon A\to f(A)$ be a partial translation on $X$. Then, $\chi_{f(A)}=\chi_{f(A)}v_fv_f^*$. As 
\[v_f^*\sigma_{h,i\beta}(\chi_{f(A)}v_f)=v_f^*e^{-\beta \bar h}\chi_{f(A)}v_fe^{\beta \bar h}=\chi_A e^{\beta (\overline{ h-h\circ f})},\]
  the KMS condition gives that 
\[\varphi(\chi_{f(A)})=\varphi(v_f^*\sigma_{h,i\beta}(\chi_{f(A)}v_f))=\varphi\Big(\chi_A e^{\beta (\overline{ h-h\circ f})}\Big).\]

Suppose now that $\varphi$ satisfies \eqref{EqVarphiSufficient}. First, notice that as $\ell_\infty(X)$ is linearly generated by the characteristic functions on $X$, this implies that      \begin{equation} \label{EqVarphiSufficient5}
        \varphi(c )
=\varphi\Big(c_{\circ f} e^{\beta (\overline{ h-h\circ f})} \Big)
\end{equation}
for all   partial translations $f$ on $X$ and all $c\in \ell_\infty(\mathrm{Im}(f))$. By abuse of notation, we extend $\varphi$ to the whole $\cstu(X)$ and still denote it by $\varphi$, i.e.,   $\varphi=\varphi\circ E$. In order to show that  $\varphi\circ E$ is a  $(\sigma_{h},\beta)$-KMS state on $\cstu(X)$, it is enough to show the KMS condition for elements of the form $av_f$, where $a\in \ell_\infty(X)$ and $f$ is a partial translation  of $X$.

Fix $a,b\in \ell_\infty(X)$ and partial translations $f$ and $g$ on $X$. Let 
\[A=\Big\{x\in \mathrm{Dom}(f)\mid f(x)\in \mathrm{Dom}(g)\text{ and }g(f(x))=x\Big\}.\]
and notice that $g\restriction f(A)=(f\restriction A)^{-1}$. We can then write
 \begin{align*}
 v_gav_f= & v_{g\restriction f(A)}av_{f\restriction A}+v_{g\restriction \mathrm{Dom}(g)\setminus f(A)}av_{f\restriction A}\\
 &+v_{g\restriction f(A)}av_{f\restriction \mathrm{Dom}(f)\setminus A} +v_{g\restriction \mathrm{Dom}(g)\setminus f(A)}av_{f\restriction \mathrm{Dom}(f)\setminus A}\\
 =&  v_{(f\restriction A)^{-1}}av_{f\restriction A}+v_{g\restriction f(A)}av_{f\restriction \mathrm{Dom}(f)\setminus A} +v_{g\restriction \mathrm{Dom}(g)\setminus f(A)}av_{f\restriction \mathrm{Dom}(f)\setminus A}.
 \end{align*}
Notice that the last two terms in the right handside of the equality above are in the kernel of the conditional expectation $E$. Therefore, 
\[E(bv_gav_f)=E(b v_{(f\restriction A)^{-1}}av_{f\restriction A}).\]
For this reason, it is enough to check the KMS condition for partial translations of $X$ which are inverse of each other. For now on, assume that $g=f^{-1}$.

Let us now show the KMS condition holds. Firstly, notice that \begin{equation}\label{EqVarphiSufficient2}
bv_fav_f^*=ba_{\circ f^{-1}}\text{ and } av_f^*e^{-\beta \bar h}bv_fe^{\beta \bar h}=a b_{\circ f}  e^{\beta (\overline{ h-h\circ f})}.
\end{equation} 
Then, letting $c=ba_{\circ f^{-1}}$, we have that  $c\in \ell_\infty(\mathrm{Im}(f))$ and 
\[c_{\circ f}=v_f^*cv_f=v_f^*bv_fav_f^*v_f=v_f^*bv_fa=ab_{\circ f}\]
Therefore,  \eqref{EqVarphiSufficient5}  gives that
\begin{align*}
\varphi(bv_fav_f^*)&=\varphi(ba_{\circ f^{-1}})\\
&=\varphi(c)\\
&=\varphi(c_{\circ f} e^{\beta (\overline{ h-h\circ f})} )\\
&=\varphi(a b_{\circ f}  e^{\beta (\overline{ h-h\circ f})} )\\
&=\varphi\Big(av_f^*e^{-\beta \bar h}bv_fe^{\beta \bar h} \Big)\\
&=\varphi( av_f^*\sigma_{ h,i\beta}(bv_f) ) .
\end{align*}
This shows that $\varphi$ is a $(\sigma_{h},\beta)$-KMS state on $\cstu(X)$.
\end{proof}

\subsection{Amenable spaces}
A priori, our flows of interest  $\sigma_{h}$ are given by any coarse map $h\colon X\to \R$ (see Proposition \ref{PropAnalyticElementsINTRO}). Therefore, being automatically coarse, bounded maps form a natural class of maps to produce flows in uniform Roe algebras. However, as we show in this subsection, the existence of KMS states for such flows reduces to the amenability of the metric space, equivalently, to the uniform Roe algebra having a positive unital trace (see \cite[Theorem 4.6]{RoeBook}). Recall:

\begin{definition}
A u.l.f.\ metric space $X$ is \emph{amenable} if there is a nonzero finitely additive measure $\mu\colon \cP(X)\to [0,\infty)$  which is \emph{invariant}, i.e.,  $\mu(A)=\mu(B)$ for all $A,B\subseteq X$ such that there is a partial translation $f\colon A\to B$. We call such measure an \emph{invariant  mean}.
\end{definition}

\begin{theorem}\label{ThmBoundedhAmenable}
Let $X$ be a u.l.f.\ metric space,  $h\colon X\to \R$ be a bounded map, and $\beta\in\R$. Then $\cstu(X)$ has  a $(\sigma_{h},\beta)$-KMS state  if and only if $X$ is amenable.
\end{theorem}

Before proving Theorem \ref{ThmBoundedhAmenable}, we isolate a  straightforward lemma which  highlights the relation between the trace and the KMS condition when the KMS state is given by elements in the \cstar-algebra.

\begin{lemma}\label{LemmaAutKMSTrace}
Let $A$ be a $C^*$-algebra and $u\in A$ be invertible. Consider the following assignments: 
\begin{enumerate}
    \item For each functional $\tau$ on $A$, let $\varphi_{\tau,u}$ be the functional given by $\varphi_{\tau,u}(a)=\tau(au)$ for all $a\in A$. 

    \item For each functional $\varphi$ on $A$, let $\tau_{\varphi,u}$ be the functional given by $\tau_{\varphi,u}(a)=\varphi(au^{-1})$ for all $a\in A$. 
\end{enumerate}
The assignment $ \tau\mapsto \varphi_\tau$ defines a bijection between the functionals $\tau $ on $A$ such that $\tau(ab)=\tau(ba)$ and the functionals $\varphi$ on $A$ such that $    \varphi(ab)=\varphi(buau^{-1}) $ for all $a,b\in A$; the inverse of this assignment  is $\varphi\mapsto \tau_{\varphi,u}$ with the appropriate domain/codomain.\qed
\end{lemma}

 \begin{proof}[Proof of Theorem \ref{ThmBoundedhAmenable}]
We start recalling a well-known fact about uniform Roe algebras: a  u.l.f.\  metric space has a positive unital trace if and only if it is amenable (\cite[Theorem 4.6]{RoeBook}). In fact, if $\mu$ is a nontrivial invariant mean on $X$, say $\mu(X)=1$, and $E\colon \cstu(X)\to \ell_\infty(X)$ is the canonical conditional expectation, then 
\[\tau(a)=\int_XE(a) d\mu,\ \text{ for all }\ a\in \ell_\infty(X),\]
defines a positive unital trace on $\cstu(X)$. On the other hand, if $\tau$ is a positive unital trace on $\cstu(X)$, then 
\[\mu(A)=\tau(\chi_A)\ \text{ for all }\ A\subseteq X\]
defines an invariant mean on $X$.

Suppose then that $X$ is amenable and that $\tau$ is the trace on $\cstu(X)$ given by a nontrivial invariant mean $\mu $ on $X$ as above. By Lemma \ref{LemmaAutKMSTrace}, $\varphi_{\tau, e^{\beta\bar  h}}$ satisfies the $(\sigma_{h},\beta)$-KMS condition. Moreover, using the formula of $\tau$, we have that 
\[\varphi_{\tau, e^{\beta \bar h}}(a)=\int_XE(a)e^{\beta \bar h} d\mu\ \text{ for all }\ a\in \cstu(X).\]
Therefore, $\varphi$ is positive and,  as $t=\sup_{x\in X} |h(x)|<\infty$, we have that 
\[\varphi_{\tau, e^{\beta \bar h}}(\chi_X)=\int_Xe^{\beta \bar h} d\mu\geq  e^{-|\beta| t} \mu(X)>0.\] Therefore, normalizing $\varphi$, we obtain a $(\sigma_{h},\beta)$-KMS state on $\cstu(X)$.
 
Suppose now that $\varphi$ is a $(\sigma_{h},\beta)$-KMS state on $\cstu(X)$. By Lemma \ref{LemmaAutKMSTrace},  $\tau_{\varphi,e^{\beta \bar h}}$  satisfies the trace condition, i.e., $\tau_{\varphi,e^{\beta \bar h}}(ab)=\tau_{\varphi,e^{\beta \bar h}}(ba)$ for all $a,b\in \cstu(X)$. As $\varphi$ is positive and factors through the canonical conditional expectation $\cstu(X)\to \ell_\infty(X)$ (Theorem \ref{ThmFactorsCondExpINTRO}), $\tau_{\varphi, e^{\beta\bar h}}$ is also positive. Finally,  it follows form our definition of $t$ that \[\tau_{\varphi,e^{\beta\bar h}}(\chi_X)=\varphi(e^{-\beta\bar  h})\geq \varphi(e^{-|\beta|t}\chi_{X})>0.\]
So, normalizing $\tau$, we obtain a positive unital trace on $X$. 
  \end{proof}

As Theorem \ref{ThmBoundedhAmenable} completely takes care of bounded maps, we can now restrict our analyses to unbounded coarse maps $h\colon X\to \R$.

\subsection{Strongly continuous KMS states}
This section deals with strongly continuous KMS states. As we  shall see below, those states are the easiest to get and, whenever they exist, they are unique (Theorem \ref{ThmStrongContKMSINTRO}). We also show that the set of $\beta$'s for which a strongly continuous KMS state exists must be either of the form $(t,\infty)$ or $[t,\infty)$, for some $t\geq 0$ (Corollary \ref{CorPropStrongContKMS} for a precise statement).

\begin{proposition}\label{PropBetaNegSingletons}
    Let $X$ be a u.l.f.\ metric space,  $h\colon X\to [0,\infty)$ be an unbounded coarse map, and $\beta<0$. If $\varphi$ is a $(\sigma_{h},\beta)$-KMS state, then $\varphi(e_{x,x})=0$ for all $x\in X$. In particular, there are no strongly continuous $(\sigma_{  h},\beta)$-KMS states on $\cstu(X)$.
\end{proposition}

\begin{proof}
    Fix $x\in X$. As $h$ is unbounded, there is a sequence $(x_n)_n$ in $X$ such that $\lim_n h(x_n)=\infty$. Then, if $\varphi$ is a $(\sigma_{  h},\beta)$-KMS state on $\cstu(X)$, we have
    \[\varphi(e_{x,x})=\varphi(e_{x,x_n}e_{x_n,x})=\varphi(e_{x_n,x}\sigma_{h,i\beta}(e_{x,x_n}))=e^{\beta(h(x_n)-h(x))}\varphi(e_{x_n,x_n}).\]
    As $(\varphi(e_{x_n,x_n}))_n$ is bounded and $\beta<0$, we conclude that $\varphi(e_{x,x})=0$ by  letting $n$ go to infinity. 
\end{proof}

\begin{proof}[Proof of Theorem \ref{ThmStrongContKMSINTRO}]
Suppose $\varphi$ is a strongly continuous $(\sigma_{ h},\beta)$-KMS state on $\cstu(X)$.  Fix an auxiliar $x_0\in X$ (this can be thought of as the ``center'' of $X$). Since all maps $(f_x\colon \{x_0\}\to \{x\})_{x\in X}$ are partial translations, the KMS condition gives us that 
\[\varphi(e_{x,x})=e^{-\beta(h(x)-h(x_0))}\varphi(e_{x_0,x_0})\]
for all $x\in X$ (see Theorem \ref{ThmSufficientKMS}). As $\varphi$ is strongly continuous, \[1=\varphi(\chi_X)=\sum_{x\in X}\varphi(e_{x,x})=e^{\beta h(x_0)}\varphi(e_{x_0,x_0})\sum_{x\in X}e^{-\beta h(x)}. \] So, $\varphi(e_{x_0,x_0})\neq 0$ and  \[Z(\beta)=\sum_{x\in X}e^{-\beta h(x)}=\frac{1}{e^{\beta h(x_0)}\varphi(e_{x_0,x_0})}\] must be finite (as well as independent on $x_0$). The formula for $\varphi$ in the statement of the theorem then follows immediately from the strong continuity of $\varphi$.

Suppose now $Z(\beta)$ is finite and $\varphi$ is given as in the  statement of the theorem.  Clearly, $\varphi$ is a strongly continuous state on $\cstu(X)$. Moreover, if $f\colon A\to B$ is a partial translation of $X$, then, by the formula of $\varphi$, we have 
\begin{align*}
    \varphi(\chi_{f(A)})&=\frac{1}{Z(\beta)}\sum_{x\in f(A)}e^{-\beta h(x)}\\
    &=\frac{1}{Z(\beta)}\sum_{x\in A}e^{-\beta h(f(x))}\\
& =   \varphi(\chi_{A}e^{\beta(\overline{ h-h\circ f})}).
\end{align*}
So, by Theorem \ref{ThmSufficientKMS}, $\varphi$ is  a $(\sigma_{h},\beta)$-KMS state on $\cstu(X)$.
\end{proof}

The following is a straightforward consequence of Proposition \ref{PropBetaNegSingletons} and Theorem \ref{ThmStrongContKMSINTRO}.

\begin{corollary}\label{CorPropStrongContKMS}
     Let $X$ be a u.l.f.\ metric space and $h\colon X\to [0,\infty)$ be an unbounded coarse map.   The subset of all $\beta\in \R$ for which there are strongly continuous $(\sigma_{  h},\beta)$-KMS states on $\cstu(X)$ is either of the form $(t,\infty)$ or $[t,\infty)$ for some $t\geq 0$.\qed
\end{corollary}

 \begin{remark}
Throughout these notes, we will see many examples for which the set of $\beta$'s admitting are strongly continuous $(\sigma_{  h},\beta)$-KMS states on $\cstu(X)$ are of the form $(\beta_0,\infty)$ for some $\beta_0>0$. This could give the impression this must always be the case, however, this is not so. For instance, let $X=\{n\in\N\mid n\geq 3\}$ and let $h(x)=\log(x\log^2(x))$ for all $x\in X$ (the restriction of $x\geq 3$ is simply so that $h$ is well defined). In this case, 
\[\sum_{n=3}^\infty e^{-\beta h(x)}=\sum_{n=3}^\infty  \frac{1}{x^\beta\log^{2\beta}(x)}\]
and this series converges if and only if $\beta\geq 1$.
 \end{remark}

\subsection{The simplest coarse space}
Under the optics of coarse geometry, the simplest infinite metric space is the \emph{coarse disjoint union of singletons}; i.e.,   any metric space which is  bijectively coarsely equivalent to 
 \[X_0=\Big\{n^2\in \N\mid n\in\N\Big\},\] 
 where $X_0$ is endowed with the usual metric $d$ on the 
  natural numbers.  
  In this subsection, we study KMS states on $X_0$.  The simplicity of the  geometry of $X_0$ makes any map $h\colon X_0\to Y$, where $Y$ is another metric space, be automatically coarse. Also,  given any $r>0$, there is a finite $F\subseteq X_0\times X_0$ such that     \[\Big\{(x,y)\in X_0\times X_0\mid d(x,y)<r\Big\}=\Big\{(x,x)\in X_0\times X_0\mid x\in X_0\Big\}\cup F.\]  Therefore, it follows that \[\cstu(X)=\ell_\infty(X)+\cK(\ell_2(X)).\]

\begin{proposition}\label{PropCoarseDisjUnionGEN}
Let $(X_0,d)$ be the coarse disjoint union of singletons described above. If $\varphi$ is a state on $\ell_\infty(X_0)$ such that $\varphi\restriction c_0(X_0)=0$, then $\varphi\circ E$ is a $(\sigma_{h},\beta)$-KMS state on $\cstu(X_0)$ for all $h\colon X_0\to \R$ and all $\beta\in \R$; where $E\colon \cstu(X_0)\to \ell_\infty(X_0)$ denotes the canonical conditional expectation.
\end{proposition}

\begin{proof}
   Let $f\colon A\subseteq X_0\to B\subseteq X_0$ be a partial translation. Then, there must be a partition $A=A_1\sqcup A_2$ such that  $f(x)=x$ for all $x\in A_1$ and $|A_2|<\infty$.   As $\varphi\restriction c_0(X_0)=0$,   we have that
   \[\varphi(\chi_{f(A)})=\varphi(\chi_{f(  A_1)}+\chi_{f( A_2)})=\varphi(\chi_{f(  A_1)})=\varphi(\chi_{  A_1}).\]
  Similarly, we have  
   \[\varphi\Big(\chi_Ae^{\beta(\overline{ h-h\circ f})}\Big)=\varphi\Big(\chi_{  A_1}  e^{\beta(\overline{ h-h\circ f})}\Big)=\varphi(\chi_{  A_1}).\]
The result then follows from  Theorem \ref{ThmSufficientKMS}.
\end{proof}

\begin{remark}
Here is   a  more conceptual way of obtaining Proposition \ref{PropCoarseDisjUnionGEN}: notice that since $\cstu(X_0)=\ell_\infty(X_0)+\cK(\ell_2(X_0))$, we must have  $\roeq(X_0)\cong \ell_\infty/c_0$; so, $\roeq(X_0)$ is abelian. Moreover, as $\sigma_h$ is the identity on $\ell_\infty(X_0)$, the flow $\sigma_h^\infty$ induced by $\sigma_h$ on $\roeq(X_0)$ is trivial (see Section \ref{SectionIntro} for the definition of $\sigma_h^\infty$). In particular, any state on $\roeq(X_0)$ is KMS for any $\beta$. The result is then a corollary of Proposition \ref{PropRoeqQStatesAffineIsoINTRO}.
\end{remark}

We now restrict our study of KMS states on $X_0$ to a specific map $h$. This will allow us to find all KMS states on $\cstu(X_0)$ for   the corresponding flow. For the sake of generality, we first isolate a result which does not depend on $X$ being the coarse disjoint union of singletons per se.

\begin{corollary}\label{CorhIsLog}
Let $d$ be any u.l.f.\ metric on $\N$ for which   the map   $h(x)=\log(x)$ is coarse and let $\beta\in \R$. If $\varphi$ is a  strongly continuous $(\sigma_{h},\beta)$-KMS state on $\cstu(\N,d)$, then $\beta>1$ and   
\begin{equation}\label{EqProphIsLog}\varphi([a_{x,y}])=\frac{1}{\sum_{n=1}^\infty \frac{1}{n^\beta}}\sum_{n=1}^\infty \frac{a_{x,x}}{n^\beta},
\end{equation}
for all $[a_{x,y}]\in \cstu(\N,d)$.
\end{corollary}

\begin{proof}
This is a straightforward consequence of Theorem \ref{ThmStrongContKMSINTRO}.
\end{proof}

We can now describe the KMS states on $X_0$ completely with $h=\log$. Precisely:

\begin{corollary}\label{CorCoarseDisLog}
Let $X_0=\{n^2\mid n\in \N\}$ be the coarse disjoint union of singletons described above,  $\beta\in \R$, and $h\colon X\to \R$ be given by  $h(x)=\log(\sqrt{x})$ for all $x\in X_0$. The $(\sigma_{h},\beta)$-KMS states of $\cstu(X_0)$ are precisely the following:
\begin{enumerate}
    \item Any state on $\cstu(X_0)$ which vanishes on $c_0(X_0)$,
    \item If $\beta>1$, then  $\cstu(X_0)$ has a unique strongly continuous  $(\sigma_{h},\beta)$-KMS state and this state is given by
\[\varphi([a_{x,y}])=\frac{1}{\sum_{n=1}^\infty \frac{1}{n^\beta}}\sum_{n=1}^\infty \frac{a_{n^2,n^2}}{n^\beta},\]
for all $[a_{x,y}]\in \cstu(X_0)$, and 
\item for $\beta>1$,  any convex combination of the states above.
\end{enumerate}   
\end{corollary}

\begin{proof}
 This follows immediately from Propositions \ref{PropCoarseDisjUnionGEN} and  Corollary \ref{CorhIsLog}.
\end{proof}

\section{Intermission}\label{SectionIntermission}

As seen in Theorem \ref{ThmStrongContKMSINTRO}, strongly continuous KMS states on uniform Roe algebras are completely understood; so we are left to understand the  strongly discontinuous case. In this section, before  explicitly perusing this goal,  we take a short break from uniform Roe algebras per se, and present some results about KMS states on arbitrary \cstar-algebras with respect to arbitrary flows. The technical results herein will be essential in the analysis to follow of KMS states on uniform Roe algebras which are   strongly discontinuous.

We start by properly stating the settings of this section. But firstly, we recall some standard notation:  if $A$ is a \cstar-algebra, then $\cZ(A)$ denotes the \emph{center of $A$}, i.e., 
\[\cZ(A)=\{b\in A\mid ab=ba\}.\]
Moreover, if $K$ is a   compact Hausdorff space, then $C(K)$ denotes the \cstar-algebra of all continuous functions $K\to \C$.

\begin{assumption}\label{Assumption}
Throughout this section, we fix a unital \cstar-algebra    $A$,  a flow $\sigma$  on $A$, and $\beta\in \R$. Moreover, we fix a  unital \cstar-subalgebra $C\subseteq A$ contained in $\cZ(A)$, and identify $C$ with $ C(\Omega(C))$ via the Gelfand transform; here $\Omega(C)$ denotes  the spectrum of $C$.
\end{assumption}

\begin{proposition}\label{PropCutByHInCenter}
In the setting of Assumption \ref{Assumption}: If $\varphi$ is a $(\sigma,\beta)$-KMS state on $A$ and   $c\in A$ is a positive element in the center of $A$ with $\varphi(c)\neq 0$, then 
the state $\varphi_c$ on $A$ defined by \[\varphi_c(a)=\frac{\varphi(ac)}{\varphi(c)}, \text{ for all } a\in A,\]
is a $(\sigma,\beta)$-KMS state on $A$.
\end{proposition}

\begin{proof}
First notice that, as $c\in \cZ(A)$, then $ac$ is positive for all positive $a\in A$. Therefore, $\varphi_c$ is indeed a state. Given $a,b\in A$, with $b$ analytic, we have 
\[\varphi_c(a\sigma_{i\beta}(b))=\frac{\varphi(a\sigma_{i\beta}(b)c)}{\varphi(c)}=\frac{\varphi(ac\sigma_{i\beta}(b))}{\varphi(c)}=\frac{\varphi(bac)}{\varphi(c)}=\varphi_h(ba).\]
So, $\varphi_c$ is a $(\sigma,\beta)$-KMS  state on $A$. 
\end{proof}

\begin{proposition}\label{PropRestrictCenterVanishes}
In the setting of Assumption \ref{Assumption}: If $\varphi$ is an extreme $(\sigma,\beta)$-KMS state on $A$, then there is $x\in \Omega(C)$ such that 
\[\varphi(a)=a(x)\ \text{ for all }\ a\in C=C(\Omega(C)).\]
In particular, letting \[J_x=\{a\in C(\Omega(C))\mid a(x)=0\},\] we have that    $\varphi \restriction J_x=0$.
\end{proposition}

\begin{proof}
By Riesz representation theorem, there is a probability measure $\mu$ on $\Omega(C)$ such that 
\[\varphi(a)=\int_{\Omega(C)}ad\mu\ \text{ for all } a\in C.\]
Let $K\subseteq \Omega(C)$ be the support of $\mu$.  Let us show that $K$ is a singleton. In order to prove this, suppose by contradiction that there are two distinct points $x,y\in K$. By Urysohn's lemma, we can pick a positive  $k\in C(\Omega(C))$ with $\|k\|\leq 1$ and such that $k(x)=1$ and $k(y)=0$. Setting $\ell=1-k$, we have that both $k$ and $\ell$ are not identically zero on $K$, so both $\varphi(k)$ and $\varphi(\ell)$ are nonzero.  By Proposition \ref{PropCutByHInCenter}, $\varphi_k$ and $\varphi_\ell$ are $(\sigma,\beta)$-KMS states on $A$, and it is clear that 
\[\varphi=\lambda\varphi_k+(1-\lambda)\varphi_\ell,\]
where $\lambda=\varphi(k)$. Since  $\varphi_k\neq \varphi_\ell$, this contradicts the assumption that $\varphi$ is an extreme $(\sigma,\beta)$-KMS state. So, $K$ contains only one point, say $K=\{x\}$. Therefore, $\mu$ must be the dirac measure on $\{x\}$, which gives that 
\[\varphi(a)=a(x), \text{ for all }a\in C(\Omega(C)).\]

The last claim follows straightforwardly from the above. 
\end{proof}

\begin{definition}
In the setting of Assumption \ref{Assumption}: 
\begin{enumerate}
\item We denote the set of all $(\sigma,\beta)$-KMS states on $A$ by $K_\beta$. 
\item    For each $x\in \Omega(C)$, let 
\[J_x=\{a\in C=C(\Omega(C))\mid a(x)=0\} \text{ and }K^x_\beta=\{\varphi\in K_\beta\mid \varphi\restriction J_x=0\}.\]
\end{enumerate}
\end{definition}

 It is plainly clear that each $K^x_\beta$ is a weak$^*$-closed convex subset of $K_\beta$.

\begin{proposition}\label{PropKextremeSubset}
In the setting of Assumption \ref{Assumption}: For all $x\in \Omega(C)$, one has that $K^x_\beta$ is an extreme subset of $K_\beta$.
\end{proposition}

\begin{proof}
Pick $\varphi\in K^x_\beta$ and assume that 
\[\varphi=\lambda\varphi_1+(1-\lambda)\varphi_2,\]
where $\varphi_1,\varphi_2\in K_\beta$ and $\lambda\in (0,1)$. Denoting by $\psi$, $\psi_1$, and $\psi_2$ the restrictions of $\varphi$, $\varphi_1$, and $\varphi_2$ to $C$, respectively, it is apparent that 
\[\psi=\lambda\psi_1+(1-\lambda)\psi_2.\]
By Proposition \ref{PropRestrictCenterVanishes}, $\psi$ is a character of $C=C(\Omega(C))$. Hence, $\psi$ is an extreme point of the unit ball of the dual of $C$. This shows that $\psi=\psi_1=\psi_2$, which in turn implies that both $\varphi_1$ and $\varphi_2$ vanish on $J_x$. Therefore, $\varphi_1,\varphi_2\in K^x_\beta$ as desired. 
\end{proof}

We can now present the main result of this section. In it, $\mathrm{Ext}(K_\beta)$ (resp. $\mathrm{Ext}(K_\beta^x)$) denotes the subset of all extreme elements of $K_\beta$ (resp. $\mathrm{Ext}(K_\beta^x)$).

\begin{theorem}\label{ThmMainIntermission}
In the setting of Assumption \ref{Assumption}: We have
\[\mathrm{Ext}(K_\beta)=\bigsqcup_{x\in \Omega(C)}\mathrm{Ext}(K^x_\beta).\] Moreover, if there is a $(\sigma,\beta)$-KMS state on $A$ whose restriction to $C$ is faithful, then $K^x_\beta\neq \emptyset$ for all $x\in \Omega(C)$.  In particular, if such KMS state exists, we have that 
\[|\mathrm{Ext}(K_\beta)|\geq |\Omega(C)|. \]
\end{theorem}

 \begin{proof}
By Proposition \ref{PropRestrictCenterVanishes}, every extreme  point $\varphi$  of $K_\beta$ lies in some $K^x_\beta$ and, in this case, $\varphi$ is evidently an extreme point of $K^x_\beta$. Conversely, as each $K^x_\beta$ is an extreme subset of $K_\beta$ (Proposition \ref{PropKextremeSubset}), every extreme point of any $K^x_\beta$ is an extreme point of $K_\beta$.

Suppose now that there is a $(\sigma,\beta)$-KMS state $\varphi$ on $A$ whose restriction to $C$ is faithful. Fix $x\in \Omega(C)$ and let us show $K^x_\beta\neq \emptyset$. Let $\cV$ be the family of all open subsets of $\Omega(C)$ which contain $x$ and, for each  $V\in \cV$, let $h_V\colon \Omega(C)\to [0,1]$ be a continuous function such that $h_V(x)=1$ and $h_V(y)=0$ for all $y\not\in V$. By the faithfulness of $\varphi$, $\varphi(h_V)\neq 0$ for all $V\in \cV$. Therefore, by Proposition \ref{PropCutByHInCenter}, each $\varphi_V=\varphi_{h_V}$ is a $(\sigma,\beta)$-KMS state on $A$.

Consider  $\cV$ as a directed set with the usual reverse containment order. By Banach-Alaoglu theorem,  $K_\beta$ is weak$^*$-compact. Hence,   by passing to a subset if necessary, we can assume that $(\varphi_V)_{V\in \cV}$ converges to some $\psi\in K_\beta$ in the weak$^*$-topology. As $\psi$ is a limit of $(\varphi_V)_{V\in \cV}$ and as $\lim_{V,\cV}\|a h_V\|=0$, for all $a\in J_x$, 
the state $\psi$ must vanish on $J_x$. This shows that $\varphi\in K^x_\beta$ and $K^x_\beta$ cannot be empty as desired.

The last claim is a straightforward consequence of  the above. 
\end{proof}

\section{Factoring KMS states through the uniform Roe corona}\label{SectionURCandHC}

In this section, we return to the setting of uniform Roe algebras and study  strongly discontinuous KMS states (the strongly continuous case was already completely treated in Theorem \ref{ThmStrongContKMSINTRO}). We start noticing that, in order to study such states,  it is enough to study the KMS states  which vanish on the ideal of compact operators. Precisely:

\begin{proposition}\label{PropKMSStateVanishComp}
Let $X$ be a u.l.f.\ metric space, $h\colon X\to \R$ be coarse, and $\beta\in \R$.   Suppose $\varphi$ is a $(\sigma_{h},\beta)$-KMS state on $\cstu(X)$ and define a positive functional $\psi$ on $\cstu(X)$ by letting 
\[\psi(a)=\lim_{F,\cF} \sum_{x\in F}a_{x,x}\varphi(e_{x,x})\ \text{ for all }\ a=[a_{x,y}]\in \cstu(X),\] where   $\cF$ is the net of all  finite subsets of $X$ ordered by reverse inclusion. Then, $\psi$ is well defined and 
\begin{enumerate}
\item $\psi$ is  strongly continuous and satisfies the  $(\sigma_{h},\beta)$-KMS condition, and
\item  $\varphi-\psi$ is a positive functional which satisfies the   $(\sigma_{h},\beta)$-KMS condition and vanishes on $\cK(\ell_2(X))$. 
\end{enumerate}
\end{proposition}

\begin{proof}
The fact that $\psi$ is well defined follows straightforwardly from the fact that $\varphi$ is positive and factors through $\ell_\infty(X)$ (Theorem \ref{ThmFactorsCondExpINTRO}). Positivity and strong continuity of $\psi$ are  then completely immediate. It is also immediate that $\psi\leq \varphi$, so $\varphi-\psi$ is also positive. Since $\psi\restriction \cK(\ell_2(X))=\varphi\restriction \cK(\ell_2(X))$, $\varphi-\psi$ vanishes on the compacts. We are only left to show that both $\psi$ and $\varphi-\psi$ satisfy the $(\sigma,\beta)$-KMS condition. But this is an immediate consequence of Theorem \ref{ThmSufficientKMS} and the formula of $\psi$.
\end{proof}

 Theorem \ref{ThmStrongContKMSINTRO} and Proposition \ref{PropKMSStateVanishComp}  show that, in order to understand the  KMS states on uniform Roe algebras, we only need to focus of the states which   vanish on the ideal of compact operators. For the remainder of this section, this will be our focus. Since the compacts form an ideal, we can factor those states through the quotient algebra. For that, recall that the \emph{uniform Roe corona of $X$} is 
 \[\roeq(X)=\cstu(X)/\cK(\ell_2(X))\]
 (see Definition \ref{DefintionRoeq}). If $\varphi$ is a state on $\cstu(X)$ which vanishes on $\cK(\ell_2(X))$, then $\varphi$ gives rise to a well-defined  state $\psi$ on $\roeq(X)$ determined by 
\[\psi(\pi(a))=\varphi(a), \text{ for all }\ a\in \cstu(X).\]
Moreover,  given a coarse map $h\colon X\to \R$, the flow $\sigma_{h}$ induces a flow $\sigma_{h}^\infty$ on $\roeq(X)$ by letting
\[\sigma_{h,t}^\infty(\pi(a))=\pi(\sigma_{h,t}(a))\ \text{ for all }\ a\in \cstu(X)\ \text{ and all }\ t\in \R\]
(see  Subsection \ref{SubsectionIntroMain} for more details).

Proposition \ref{PropRoeqQStatesAffineIsoINTRO} highlights the relations between $\varphi$ and $\psi$, and $\sigma_{h}$ and $\sigma_{h}^\infty$ defined above.

\begin{proof}[Proof of Proposition \ref{PropRoeqQStatesAffineIsoINTRO}]
Notice that if $b$ is an analytic element in $\cstu(X)$ for $\sigma_{h}$, then $\pi(b)$ is analytic for $\sigma_{h}^\infty$ and, moreover,
\[\pi(\sigma_{h,z}(b))=\sigma_{h,z}^\infty(\pi(b))\ \text{ for all }\ z\in \C.\]
Therefore, the image of the set of all analytic elements in $\cstu(X)$ under $\pi$ forms a dense set of analytic elements in $\roeq(X)$. Consequently, in order to check that a state $\psi$ on $\roeq(X)$ is a $(\sigma_{h}^\infty,\beta)$-KMS state, it suffices to prove that 
\[\psi(\pi(a)\sigma_{h,i\beta}^\infty(\pi(b)))=\psi(\pi(b)\pi(a)),\]
for all $a,b\in A$ with $b$ analytic. Observing that the left-hand-side above coincides with $(\psi\circ\pi)(a\sigma_{h,i\beta}(b))$ and that the right-hand-side equals $(\psi\circ\pi)(ba)$, the first statement of the proposition follows. The second statement in turn follows from the first one immediately.
\end{proof}

Proposition \ref{PropRoeqQStatesAffineIsoINTRO} then reduces our problem to the one of understanding the KMS states on the uniform Roe corona $\roeq(X)$. In view of Section  \ref{SectionIntermission}, it will be useful to study the center $\roeq(X)$  as well as its \cstar-subalgebras. This brings up a seemingly unexpected link between KMS states and the \emph{Higson corona}.  Recall:

\begin{definition}\label{DefinitionHigsonCor}
Let $X$ be a u.l.f.\ metric space. 
\begin{enumerate}
\item A bounded function $f\colon X\to \C$ is a \emph{Higson function} if for all $\eps>0$ and all $R>0$ there is a finite $F\subseteq X$ such that 
\[\forall x,y\in X\setminus F, \ d(x,y)<R\ \text{ implies }\ |f(x)-f(y)|<\eps.\]
The set of all Higson functions on $X$ forms a \cstar-subalgebra of $\ell_\infty(X)$ which we denote by $C_h(X)$.
\item The spectrum  of $C_h(X)$, denoted by $hX$, is called the \emph{Higson compactification} of $X$. So, the Gelfand  transform gives us the identification $C(hX)=C_h(X)$.
\item The boundary $\nu X=hX\setminus X$ is called the \emph{Higson corona} and we have the identification $C(\nu X)=C_h(X)/c_0(X)$.
\end{enumerate}
\end{definition}

Notice that, as $C_h(X)\subseteq \ell_\infty(X)$, we may canonically view $C_h(X)/c_0(X)$ as a \cstar-subalgebra of $\roeq(X)$; so, by the identification   $C(\nu X)=C_h(X)/c_0(X)$, we have  \[C(\nu X)\subseteq \roeq(X).\]  It has been recently observed that the center of $\roeq(X)$ is precisely the Higson corona of $X$. Indeed, the following was proven in  \cite[Proposition 3.6]{BaudierBragaFarahVignatiWillett2023} as a consequence of  \cite[Theorem 3.3]{SpakulaZhang2020JFA}.

\begin{proposition}\label{PropHigsonCorCenter}
Given a u.l.f.\ metric space $X$, we have that \[C(\nu X)=\cZ(\roeq(X)).\]
\end{proposition}

 We now apply our results of Section \ref{SectionIntermission} to our coarse setting. In what follows, if $C$ is a unital \cstar-algebra, $\Omega(C)$ denotes the spectrum of $C$. So, $\Omega(C)$ is a compact Hausdoff topological space and  we use the identification $C=C(\Omega(C))$ given by   Gelfand transform.

\begin{theorem}
 Let $X$ be a u.l.f.\ metric space, $h\colon X\to \R$ be a coarse map, and $\beta\in \R$. Let $C$ be a unital \cstar-subalgebra of $C(\nu X)$.

 \begin{enumerate}
 \item For any  extreme $ (\sigma_{h}^\infty,\beta)$-KMS state   $\psi$ on $\roeq(X)$, there is $x\in \Omega(C)$ such that \[\psi(a)=a(x) \ \text{ for all }\ a\in C=C(\Omega(C)).\]
     \item Suppose there is a $(\sigma_{h}^\infty,\beta)$-KMS state on $\roeq(X)$  whose restriction to $C$ is faithful. Then, for any $x\in\Omega(C)$, there is an extreme $(\sigma_{h}^\infty,\beta)$-KMS state $\psi$ on $\roeq(X)$  such that \[\psi(a)=a(x) \ \text{ for all }\ a\in C=(\Omega(C)).\]
 \end{enumerate}
 \label{ThmExtRoeqGENERAL}
\end{theorem}

\begin{proof} 
This is a   mere corollary of Theorem \ref{ThmMainIntermission}.
\end{proof}

\begin{proof}[Proof of Theorem   \ref{ThmExtRoeqINTRO}]
    This is a particular case of Theorem \ref{ThmExtRoeqGENERAL} with $C=C(\nu X)$.
\end{proof}

 We now obtain Theorem \ref{ThmUltimoINTRO} by proving a more general version of it. For that, we  first generalize Definition \ref{DefinitionSupp}.

 \begin{definition}\label{DefinitionSuppGENEREAL}
Let $X$ be a u.l.f.\ metric space and $\bar X$ be a compactification of $X$.
\begin{enumerate}
    \item We call $\bar X$ \emph{Higson compatible} if \[f\restriction X\in  C_h(X)\ \text{ for all }\ f\in C(\bar X).\] 
    \item If $\bar X$ is Higson compatible and  $x\in \bar X$, we say that a state    $\varphi$ on $\cstu(X)$ is \emph{$\bar X$-supported on $x$} if for all neighborhoods $U\subseteq \bar X$ of $x$, we have $\varphi(\chi_{U\cap X})=1$.
\end{enumerate} 
 \end{definition}

Notice that if $\bar X$ is a Higson compatible compactification of $X$, then $C(\bar X)$ can be canonically identified with a \cstar-subalgebra of $C_h(X)$, which in turn allows us to identify $C(\bar X)/c_0(X)$ with a \cstar-subalgebra of $C(\nu X)\subseteq \roeq(X)$.

 \begin{theorem}\label{ThmUltimoGENERAL}
 Let $X$ be a u.l.f.\ metric space,    $h\colon X\to \R$ be a coarse map, and $\beta\in \R$. Let $\bar X$ be a Higson compatible compactification of $X$.  The following holds:
 \begin{enumerate}
 \item Any extreme  $(\sigma_h,\beta)$-KMS state on $\cstu(X)$ which vanishes on the compacts is $\bar X$-supported at some element of $\bar X$.
 \item If there is a  $(\sigma_h,\beta)$-KMS state on $\cstu(X)$ which vanishes on the compacts and such that its induced state on $\roeq(X)$ is faithful on $C(\bar X)/c_0(X)$, then for every $x\in \bar X$ there is a   $(\sigma_h,\beta)$-KMS state on $\cstu(X)$ which is $\bar X$-supported on $x$.
 \end{enumerate}
 \end{theorem}

\begin{proof} 
This is    a mere corollary of Proposition \ref{PropRoeqQStatesAffineIsoINTRO} and Theorem \ref{ThmExtRoeqGENERAL}.
\end{proof}

\begin{proof}[Proof of Theorem   \ref{ThmUltimoINTRO}]
    This is a particular case of Theorem \ref{ThmUltimoGENERAL} with $\bar X=h X$.
    \end{proof}

 \subsection{The size of the Higson corona}\label{SubsectionSize}
 We show that the Higson corona of an infinite u.l.f.\ metric space must always have $2^{2^{\aleph_0}}$ many elements (Theorem \ref{ThmSizeCorona}). Together with the previous results in this section, this will give us a very strong control of the cardinality of KMS states on $\cstu(X)$.

 In this subsection, we work a lot with partial bijections $f$ of $X$ and it will be useful to be able to write ``$f(A)$'' regardless of whether $A\subseteq \mathrm{Dom}(f)$.  We then define: given any set $X$, a partial bijection $f\colon \mathrm{Dom}(f)\to \mathrm{Im}(f)$ of $X$, and $A\subseteq X$, we let 
 \[f[A]=f(A\cap \mathrm{Dom}(f)).\]
 Also, given partial bijections $f$ and $g$ of $X$, we let $g\circ f$ be the partial bijection from $ f^{-1}[\mathrm{Dom}(g)]$ to $g[\mathrm{Im}(f)]$ defined by $g\circ f(x)=g(f(x)) $ for all $x\in f^{-1}[\mathrm{Dom}(g)]$.
 
The following lemma is an easy exercise and we leave the details to the reader. 

\begin{lemma}\label{LemmaSetEquaPartialFunc}
    Let $f$ and $g$ be partial bijections of $X$. Then
    \[f[A]\cap g[B]=g^{-1}((g^{-1}\circ f)[A]\cap B)\]
    for all $A,B\subseteq X$.\qed
\end{lemma}

\begin{definition}
    Let $X$ be a u.l.f.\ metric space. 
A subset $A\subseteq X$ is \emph{thin} if $f[A]\cap A$ is finite for all partial translations $f$ of $X$ which \emph{do not fix points}, i.e.,  such that $f(x)\neq x$ for all $x\in \mathrm{Dom}(f)$.
\end{definition}

 \begin{lemma}\label{LemmaThinSets}
     Every infinite u.l.f.\  metric space  contains an infinite thin subset.
 \end{lemma}

 \begin{proof}
If  $(X,d)$ is infinite and u.l.f., then $X$ is unbounded. Hence, we can   inductively pick a  sequence $(x_i)_{i\in \N}$ in $X$ such that 
\[d(x_k,x_\ell)\geq \max _{i,j<\ell}d(x_i,x_j)+\ell\]
for all $\ell>k$ in $\N$. The set $A=\{x_i\mid i\in \N\}$ is clearly thin.\footnote{Equivalently, if $A\subseteq X$ is the image of a coarse embedding of $\{n^2\mid n\in\N\}$ in $X$, then $A$ is thin.}
 \end{proof}

 \begin{proposition}\label{PropThinSets}
     Let $X$ be a u.l.f.\ metric space,  $C\subseteq X$ be thin, and let $C=A\sqcup B$ be a partition of $C$. If $f$ and $g$ are partial translations of $X$, then $f[A]\cap g[B]$ is finite.
 \end{proposition}

  \begin{proof}
By Lemma \ref{LemmaSetEquaPartialFunc}, it is enough to show that $(g^{-1}\circ f)[A]\cap B$ is finite. As the composition of partial translations is still a partial translation, it is enough to show that  $ f[A]\cap B$  is finite for any partial translation $f$ of $X$. Fix such $f$ and, replacing $A$ with $A\cap \mathrm{Dom}(f)$, we also assume that $A\subseteq \mathrm{Dom}(f)$. Let us show that $f(A)\cap A$ is finite. Set
\[A_0=\{x\in A\mid f(x)=x\} \ \text{ and }\ A_1=A\setminus A.\]
Then, as $A\cap B=\emptyset$, we have
\[f(A)\cap B=f(A_0\cup A_1)\cap B=f(A_1)\cap B.\]
Let $f_1=f\restriction A_1$. Then $f_1$ has no fixed points and 
\[f_1(A)\cap B\subseteq f_1[C]\cap C.\]
Since $C$ is thin, $f_1[C]\cap C$ must be finite. So,  $f(A)\cap B$ is finite.
 \end{proof}
 
 Given a u.l.f.\ metric space $X$, let  $\beta X$ denote the \emph{Stone--\v{C}ech compactification of $X$}.\footnote{Please be careful not to mistake this $\beta$ for the inverse temperature!} Since $X$ is discrete, $\beta X$ can be identified with the space of ultrafilters on $X$ endowed with the \emph{Stone topology}, i.e., the topology generated by open sets of the form 
 \[U_A=\{\omega\in \beta X\mid A\in \omega\}\]
 for all $A\subseteq X$.  Given $A\subseteq X$, we let $\bar A$ denote the closure of $A$ in $\beta X$ and let $\hat A=\bar A\setminus A$. By the defining property of $\beta X$, any element in   $\ell_\infty(X)$ extends to one in $C(\beta X)$. This defines a  canonical isomorphism between $\ell_\infty(X)$ and $C(\beta X)$, and we identify those algebras under this isomorphism. We identify $C(\hat X)$ with $C(\beta X)/c_0(X)$ via Gelfand transform. Hence,  under these identifications, we have
 \[C(\hat X)=\ell_\infty(X)/c_0(X)\subseteq\roeq(X). \]  
 
We now define invariant subsets of the Stone--\v{C}ech compactification.   For that, recall that, by the defining property of $\beta X$, any partially defined map $f\colon \mathrm{Dom}(f)\subseteq X\to \mathrm{Im}(f)\subseteq X$ can be continuously extended to a map $\overline{\mathrm{Dom}(f)}\to \overline{\mathrm{Im}(f)}$. By abuse of notation, we still denote this map by $f$.

\begin{definition}\label{DefiInvariant}
    Let $X$ be a u.l.f.\ metric space and $A\subseteq \beta X$. We say that $A$ is \emph{invariant} if $f[A]\subseteq A$ for all partial translations $f$ of $X$.
\end{definition}

For the next lemma, notice that if  $L\subseteq \hat X$ is a clopen subset, then  $\chi_L\in C(\hat X)$. Hence, it makes sense to wonder whether $\chi_L$ can also be in $C(\nu X)\subseteq C(\hat X)$.

\begin{lemma}\label{LemmaInvariantCorona}
    Let $X$ be a u.l.f.\ metric space and $L\subseteq \hat X$ be an  invariant clopen subset. Then $\chi_L\in C(\nu X)$.
\end{lemma}

\begin{proof}
 By Proposition
\ref{PropHigsonCorCenter}, it is enough to notice that $\chi_L$ is in the center of $\roeq(X)$. Hence, since $\cstu[X]$ is dense in $\cstu(X)$ and spanned by $av_f$, where $a\in \ell_\infty(X)$ and $f$ is a    partial translation of $X$, we only need to show that $\chi_{L}$ commutes with $w_f=\pi(v_f)$ for all partial translations $f$ of $X$. Fix such partial translation $f$ and let $A=\mathrm{Dom}(f)$ and $B=\mathrm{Im}(f)$. Then, $w_f=  \chi_{\hat B} w_f\chi_{\hat A}$ and 
\begin{equation}\label{EqCenter}
    w_f\chi_L=  w_f\chi_{\hat A\cap L}=\chi_{\hat B\cap f[L]}w_f 
    =\chi_{f[L]}w_f
\end{equation} 
Notice that $f[L]=\hat B\cap L$. Indeed, since $L$ is invariant and $f$ is a partial translation, $f[L]\subset \hat B\cap L$. On the other hand, as $f^{-1}$ is also a partial translation, we have $f^{-1}[L]\subseteq L$. Hence, as $\hat B\cap L\subseteq f[f^{-1}[L] ]$, we also have $ \hat B\cap L\subseteq f[L]$. We can then conclude from \eqref{EqCenter} that $w_f\chi_L=\chi_Lw_f$. As the partial translation $f$ was arbitrary, we conclude that $\chi_L\in C(\nu X)$ as desired.
\end{proof}

\begin{lemma}\label{LemmaSepHigCo}
    Let $X$ be a u.l.f.\ metric space and $C\subseteq X$ be thin. If $\omega,\omega'\in \hat C$ are distinct, then there are disjoint invariant open subsets $U,V\subseteq \hat X$ such that $\omega\in U$ and $\omega'\in V$.
\end{lemma}

\begin{proof}
Since $\omega,\omega'\in \bar C$, it follows that $C\in \omega$ and $C\in \omega'$. As $\omega'\neq \omega'$, there is $D\subseteq X$ such that $D\in \omega $ and $D\not\in \omega'$. Hence, 
\[A=C\cap D\in \omega\ \text{ and }\ B=C\setminus D\in \omega'.\]
Therefore, $\omega\in \hat A$ and $\omega'\in \hat B$. Let $\mathcal{PT}$ denote the set of all partial translations of $X$ and define
\[U=\bigcup_{f\in \mathcal{PT}}\widehat{f[A]}\ \text{ and }\ V=\bigcup_{f\in \mathcal{PT}}\widehat{f[B]}.\]
Clearly, $U$ and $V$ are open, invariant and contain $\omega$ and $\omega'$, respectively. We only need to notice they are also disjoint. For that, notice that Lemma \ref{LemmaThinSets} implies that $f[A]\cap g[B]$ is finite for all $f,g\in \mathcal {PT}$. But then $\widehat{f[A]}\cap \widehat{g[B]}=\emptyset$  for all $f,g\in \mathcal {PT}$, which in turn implies that $U\cap V=\emptyset$.
\end{proof}

\begin{theorem}
    \label{ThmSizeCorona}
    Let $X$ be an infinite u.l.f.\ metric space. Then $\nu X$ has at least $2^{2^{\aleph_0}}$ elements.
\end{theorem}

\begin{proof}
Let $p\colon \hat X\to \nu X$ be the continuous surjection such that the canonical identification of $C(\nu X)$ with a \cstar-subalgebra of $C(\hat X)$ is given by the map
\[a\in C(\nu X)\mapsto a\circ p\in C(\hat X
).\]
Let $C\subseteq X$ be an infinite thin subset given by Lemma \ref{LemmaThinSets}. As $\hat C$ is the set of all nonprincipal ultrafilters on $C$ and $C$ is countable, we have that $|C|=2^{2^{\aleph_0}}$. Therefore, in order to obtain that $\nu X$ has $2^{2^{\aleph_0}}$ elements, it is enough to show that  $p$ is injective on $\hat C$.

Let $\omega,\omega'\in \hat C$ be distinct. By  Lemma \ref{LemmaSepHigCo}, there are disjoint  invariant open subsets $U,V\subseteq \hat X$   containing $\omega$ and $\omega'$, respectively. As $\beta X$ is extremely disconnected, $\bar U$ is clopen in $\hat X$ which implies that  the characteristic function of $\bar U$, $\chi_{\bar U}$, is a continuous function in $C(\hat X)$.  As $\bar U$ is invariant, Lemma \ref{LemmaInvariantCorona} shows that $\chi_{\bar U}\in C(\nu X)$. Therefore, since we clearly have $\chi_{\bar U}(\omega)=1$ and $\chi_{\bar U}(\omega')=0$, this shows that $p(\omega)\neq p(\omega')$. 
\end{proof}

\begin{remark}
It is interesting to notice that Theorem \ref{ThmSizeCorona} is only valid for \emph{metric} u.l.f.\  spaces. Precisely, Higson coronas can be defined more generally for \emph{coarse spaces} --- for brevity, we do not define coarse spaces here, the reader can find the precise definition in \cite{RoeBook}  or \cite[Section 5]{BaudierBragaFarahKhukhroVignatiWillett2021uRaRig}.  It is known that every perfectly normal compact Hausdorff space is homeomorphic to the Higson corona of some u.l.f.\ 
 coarse space (see \cite[Page 2]{BanakhProtaso2020TOpAppli}). It is however  not surprising that the Higson corona of nonmetrizable u.l.f.\ coarse spaces can be much smaller since there will be fewer Higson functions in this case. The proof of Theorem \ref{ThmSizeCorona} cannot hold outside the metrizable world since thin sets may not exist. For instance, if $\cE_{\max}$ is the maximal u.l.f.\ coarse structure on an infinite set $X$ (see \cite[Subsection 1.3]{BaudierBragaFarahKhukhroVignatiWillett2021uRaRig} for the precise definition), then it is clear that $(X,\cE_{\max})$ has no infinite thin subsets.
\end{remark}

\begin{corollary}\label{CorSize}
    Let $X$ be an infinite u.l.f.\ metric space, $h\colon X\to \R$ be a coarse map, and $\beta\in \R$. If there is a $(\sigma^\infty_h,\beta)$-KMS state on $\roeq(X)$ whose restriction to $C(\nu X)$ is faithful, then there are $2^{2^{\aleph_0}}$ extreme   $(\sigma^\infty_h,\beta)$-KMS states on $\roeq(X)$. In particular,  there are $2^{2^{\aleph_0}}$ extreme   $(\sigma_h,\beta)$-KMS states on $\cstu(X)$ which vanish on $\cK(\ell_2(X))$.
\end{corollary}
\begin{proof}
    The statement for the uniform Roe corona follows from   Theorems \ref{ThmExtRoeqGENERAL} and   \ref{ThmSizeCorona}. The statement for the uniform Roe algebra is then a consequence of Proposition \ref{PropRoeqQStatesAffineIsoINTRO}.
\end{proof}
 
\section{Applications: Branching trees}\label{SectionApli}
In this section, we apply the theory of KMS states on uniform Roe algebras developed above to $n$-branching trees. Recall that, as mentioned in the introduction, the choice for those spaces are, in a sense, very natural. Precisely, as explained in Subsection \ref{SubsectionIntroAppli}, as long as $h\colon X\to \R$ is such that $h(x)$ is bounded below by an affine map in terms of $d(x,x_0)$ for a given $x_0\in X$, there will always be $(\sigma_{h},\beta)$-KMS states on $\cstu(X)$ for all $\beta>0$ as long as $X$ has polynomial growth. Therefore, in order to find more interesting phase transitions, it is natural to look at metric spaces with exponential growth.

\subsection{$n$-branching trees}
Given a set $\Gamma$, we let $\Gamma^{<\infty}$ be the set of all finite words on $\Gamma$, including the empty word; which we denote by $\emptyset$. In other words,  if  $\gamma\in \Gamma^{<\infty}$, then either $\gamma=\emptyset$ or $\gamma=(\gamma_1,\ldots, \gamma_n)$ for some $n\in\N$ and some $\gamma_1,\ldots, \gamma_n\in \Gamma$. Given $\gamma\in \Gamma^{<\infty}$, if $\gamma=\emptyset$, we say that the \emph{length of $\gamma$} is $0$, if $\gamma=(\gamma_1,\ldots, \gamma_n)$, we say that the \emph{length of $\gamma$} is $n$; either way, we denote the length of $\gamma$ by $|\gamma|$ and we write $\gamma=(\gamma_1,\ldots, \gamma_{|\gamma|})$ (here it is understood that if $|\gamma|=0$, then $\gamma=\emptyset$). Given $\gamma,\gamma'\in \Gamma^{<\infty}$ we denote the \emph{concatenation of $\gamma$ and $\gamma'$} by $\gamma^\smallfrown \gamma'$, i.e., \[\gamma^\smallfrown \gamma'=(\gamma_1,\ldots, \gamma_{|\gamma|},\gamma'_1,\ldots, \gamma'_{|\gamma'|}).\]

 \begin{definition}
Let $n\in\N$ and consider $\Gamma=\{1,\ldots, n\}$. We make $\Gamma^{<\infty}$ into a graph by saying that any two distinct elements $\gamma,\gamma'\in \Gamma^{<\infty}$ are adjacent if there is $k\in \Gamma$ such that either $\gamma_1=\gamma_2^\smallfrown k$ or $\gamma_2=\gamma_1^\smallfrown k$. This defines a graph structure on $\Gamma^{<\infty}$ making it into a connected (undirected) graph. We can then see $\Gamma^{<\infty}$ as a metric space endowed with the shortest path distance. We call this metric space the \emph{$n$-branching tree} and denote it by $T_n$.  
 \end{definition}
 
For simplicity, we now isolate   the setting of this subsection.

\begin{assumption}\label{AssumptionBranchingTree}
Let  $n\in\N$ and let $T_n$ be the $n$-branching tree endowed with the shortest path metric, denoted by $d$. Let $h\colon T_n\to \R$ be the function given by  $h(x)=d(x,\emptyset)$  for all $x\in T_n$
\end{assumption} 
 
 \subsection{Strongly continuous KMS states on $\cstu(T_n)$}
 We start with a simple lemma about states on $\ell_\infty$. In the next lemma, $\ell_\infty=\ell_\infty(\N)$ and $c_0=c_0(\N)$.
 
\begin{lemma}\label{LemmaStrongCont}
    Let $\varphi$ be a state on $\ell_\infty$. If $\varphi\restriction c_0$ has norm 1, then $\varphi $ is strongly continuous.
\end{lemma}

\begin{proof} Let  $(\ell_\infty)_+$  denote the subset of $\ell_\infty$ containing only sequences of positive numbers. Let $a=(a_k)_k\in (\ell_\infty)_+$ have norm at most 1. Then, as $\varphi$ is positive, we have that 
\begin{equation}\label{EqLemmaStateSOT1}
\varphi(  a)\geq \sum_{j=1}^k a_j\varphi(\chi_{\{j\}})\ \text{ for all }\ k\in\N.
\end{equation}
For the same reason, we have \[
\varphi(\chi_{\N}- a)\geq \sum_{j=1}^k (1- a_j)\varphi(\chi_{\{j\}})\ \text{ for all }\ k\in\N.
\]
As $\varphi\restriction c_0$ has norm 1, \[\lim_k\sum_{j=1}^k\varphi(\chi_{\{j\}})=\lim_k \varphi(\chi_{\{1,\dots,k\}})=1.\] Therefore, 
\begin{equation}\label{EqLemmaStateSOT2}
1-\varphi(  a)=\varphi(\chi_{\N})-\varphi(  a)=\varphi(\chi_{\N}-  a)\geq 1-\sum_{n=1}^ka_n\varphi(\chi_{\{n\}})
\end{equation}
for all $k\in\N$. Hence, \eqref{EqLemmaStateSOT1} and \eqref{EqLemmaStateSOT2} together imply that \[\varphi(  a)= \sum_{j=1}^\infty a_j\varphi(\chi_{\{j\}}).\] 

Now, for an arbitrary $a\in \ell_\infty$, splitting  $a$ into its positive and negative parts, the previous paragraph imply that $\varphi( a)= \sum_{j=1}^\infty a_j\varphi(\chi_{\{j\}})$, so the lemma follows. 
\end{proof}

 The next result is a partial version of Theorem \ref{ThmnBranchingTreeCOMPLETEINTRO}.

\begin{theorem}\label{ThmnBranchingTree}
In the setting of Assumption \ref{AssumptionBranchingTree}: Given  $\beta\in \R$, there is a  $(\sigma_{h},\beta)$-KMS state on $\cstu(T_n)$ if and only if   $\beta\geq \log(n)$. Moreover,

\begin{enumerate}
    \item \label{ThmnBranchingTreeItem2}
For $\beta> \log(n)$,   there is a unique $(\sigma_{h},\beta)$-KMS state $\varphi_\beta$ on $\cstu(T_n)$  and $\varphi_\beta$ is given by 
\[\varphi_\beta([a_{x,y}])=\sum_{y\in T_n} a_{y,y}\Big(e^{-\beta|y|}-ne^{-\beta (|y|+1)}\Big)\]
for all $[a_{x,y}]\in \cstu(T_n)$. 

\item \label{ThmnBranchingTreeItem3} For $\beta=\log(n)$, the $(\sigma_{h},\beta)$-KMS states on $\cstu(T_n)$ all vanish on $\cK(\ell_2(T_n))$. 
\end{enumerate}
\end{theorem}

\begin{proof}
Suppose $\varphi$ is a $(\sigma_{h},\beta)$-KMS state on $\cstu(T_n)$. Notice that, for each $y\in T_n$, the map $f\colon T_n\to T_n$ given  by  $f(x)=x^{\smallfrown}y$, for all $x\in T_n$, is a partial translation; indeed, $d(x,f(x))=|y|$ for all $x\in T_n$. So, each $v_f$ belongs to $\cstu[T_n]$. Then, for each $y\in T_n$, we have 
\begin{align*}
\sigma_{h,i\beta}(v_f^*)&=e^{-\beta\bar h}v_f^*e^{\beta\bar h}\\
&=e^{-\beta\bar h}\Big(\SOTh\sum_{x\in X}e_{x,x^\smallfrown y}\Big)e^{\beta\bar h}\\
&=e^{\beta(d(x^\smallfrown y,\emptyset)-d( x,\emptyset))}v_f^*\\
&=e^{\beta|y|}v_f^*.
\end{align*}

 For each $y\in T_n$, set 
\[T_n^\smallfrown y=\{x\in T_n\mid x=z^\smallfrown y\text{ for some }z\in T_n\}.\]
Hence, as $\chi_{T_n}=v_f^*v_f$ and $\chi_{T_n^\smallfrown y}=v_fv_f^*$, we must have \[1=\varphi(\chi_{T_n})=\varphi(v_f^*v_f)=
\varphi(v_f\sigma_{h,i\beta}(v_f^*))=e^{\beta |y|}\varphi(\chi_{T_n^\smallfrown y})\]
for all $y\in T_n$; which implies
 \begin{equation}\label{EqVarphiAy}\varphi(\chi_{T_n^\smallfrown y})=e^{-\beta |y|} \ \text{ for all } \ y\in T_n.
 \end{equation}
Since for each $y\in T_n$, we have  \[\{y\}=T_n^\smallfrown y\setminus \bigcup_{k=1}^n T_n^\smallfrown k^\smallfrown y,\]
\eqref{EqVarphiAy} implies that 
\begin{equation}\label{EqFormulaVarphi}
\varphi(e_{y,y})=\varphi\Big(\chi_{T_n^\smallfrown y}-\sum_{k=1}^n\chi_{T_n^\smallfrown k^\smallfrown y}\Big)=e^{-\beta|y|}-ne^{-\beta (|y|+1)}
\end{equation}
for all $y\in T_n$. 

As $\varphi$ is positive, each $\varphi(e_{y,y})$ must be positive.  So, \eqref{EqFormulaVarphi}  gives that 
\[e^{-\beta|y|}\geq ne^{-\beta (|y|+1)}\ \text{ 
for all }\ y\in T_n.\]  Solving for $\beta$, this implies either $\beta=0$ or $\beta\geq \log(n)$.  If $n=1$, then $\log(n)=0$; so $\beta\geq \log(n)$. If $n>1$, then $T_n$ is not amenable; so $\cstu(T_n)$ has no positive unital traces (see \cite[Theorem 4.6]{RoeBook}). Therefore, there are no $(\sigma_{h},0)$-KMS states on $\cstu(T_n)$. In either case, we conclude that  $\beta\geq \log(n)$. Moreover, as \eqref{EqFormulaVarphi} must hold regardless of $\beta$, this also shows that the $(\sigma_{h},\log(n))$-KMS states on $\cstu(T_n)$ all vanish on $c_0(T_n)$. Since such states factors through $\ell_\infty(T_n)$ (Theorem \ref{ThmFactorsCondExpINTRO}),  \eqref{ThmnBranchingTreeItem3} follows. 

We must now show that if  $\beta\geq \log(n)$, then   $(\sigma_{h},\beta)$-KMS states exist. This will however be an immediate consequence of   \eqref{ThmnBranchingTreeItem2}. Indeed,   the set of all $\beta$'s for with $(\sigma_{h},\beta)$-KMS states exist is always  a closed set (see \cite[Proposition 5.3.25]{BratteliRobinsonBookII1997}).

We now show \eqref{ThmnBranchingTreeItem2} holds.  For this, suppose  $\beta>\log(n)$ and let us  show that any given $(\sigma_h,\beta)$-KMS state $\varphi $ must have the required form. Notice that  $\varphi \restriction \ell_\infty(T_n)$ is a state on $\ell_\infty(T_n)$. Moreover,   the computations above show that
\begin{equation}\label{Eqc0}
\varphi (a)=\sum_{y\in T_n} a_{y}\Big(e^{-\beta|y|}-ne^{-\beta (|y|+1)}\Big)
\end{equation}
for all $a=(a_y)_y\in c_0(T_n)$. Hence, 
an easy computation gives
\[\lim_{F,\cF}\varphi(\chi_F)=1,\]
where $\cF$ is the net of all finite subsets of $T_n$ ordered by reverse inclusion. Therefore, it follows that  $\|\varphi\restriction c_0(T_n)\|=1$ and, by Lemma \ref{LemmaStrongCont}, $\varphi\restriction\ell_\infty$ is strongly continuous.  This implies that  \eqref{Eqc0} holds for all   $a=(a_y)_y\in \ell_\infty(T_n)$. In other to notice that this holds for arbitrary elements of $\cstu(T_n)$,  let $E\colon \cstu(X)\to \ell_\infty(X)$ be the canonical conditional expectation and recall that,  by Theorem \ref{ThmFactorsCondExpINTRO}, we have $\varphi=\varphi\circ E$. This proves the uniqueness part of \eqref{ThmnBranchingTreeItem2}.

We are left to notice that a $\varphi$ given by the formula above is indeed a $(\sigma_{h},\beta)$-KMS state on $\cstu(T_n)$.  This will be done by using Theorem \ref{ThmSufficientKMS}.\footnote{Equivalently, this could also be done using Theorem \ref{ThmStrongContKMSINTRO}, but the computations would not be shorter.} So, let $f\colon A\to f(A)$ be partial translation on $X$. On one hand, we have that 
\[\varphi(\chi_{f(A)})=\sum_{y\in f(A)} \Big(e^{-\beta|y|}-ne^{-\beta(|y|+1)}\Big).\]
On the other hand,
 \begin{align*}
 \varphi\Big(\chi_Ae^{\beta(\overline{ h-h\circ f})}\Big)&=\sum_{x\in A} e^{\beta(|x|-|f(x)|)}\Big(e^{-\beta|x|}-ne^{-\beta(|x|+1)}\Big)\\
&=\sum_{x\in A}  \Big(e^{-\beta|f(x)|}-ne^{-\beta(|f(x)|+1)}\Big).
\end{align*}
The change of variable $y=f(x)$  give us   
\[\varphi(\chi_{f(A)})=\varphi\Big(\chi_A e^{\beta(\overline{ h-h\circ f})}\Big).\]
As $\varphi=\varphi\circ E$,  Theorem \ref{ThmSufficientKMS} gives us that $\varphi$ is a $(\sigma_{h},\beta)$-KMS state on $\cstu(X)$.
\end{proof}

\subsection{KMS states on $\cstu(T_n)$ vanishing on compacts}
In order to complete the proof of Theorem \ref{ThmnBranchingTreeCOMPLETEINTRO}, we must further analyze  the case $\beta=\log(n)$. According to Theorem \ref{ThmnBranchingTree}, the KMS states for this inverse temperature will all vanish on the ideal of compact operators and we can then make use of the material of Section \ref{SectionURCandHC}. Moreover,  ideas in   \cite[Lemma 3]{Chou1969PAMS} will also be extremely useful in order to compute to the precise cardinality of the set of extreme $(\sigma_h,\beta)$-KMS states on $\cstu(T_n)$.

\subsubsection{Precise cardinality of the set of KMS states on $\cstu(T_n)$ for $\beta=\log(n)$} We start by setting up some notation.  Given $y\in T_n$,   consider the map
 \begin{align*}\tilde y\colon T_n&\to \beta T_n\\
 x&\mapsto x^\smallfrown y.
 \end{align*}
 Then, by the defining property of $\beta T_n$, $\tilde y$ can be extended to a  continuous map $\beta T_n\to \beta T_n$ which, by abuse of notation, we still denote by $\tilde y$. Notice that  
 \[\overline{\tilde y(A)}=\tilde y(\bar A) \ \text{ for all }\ A\subseteq T_n,\]
where the closures above are taken in $\beta T_n$ (see Lemma \cite[Lemma 2.1]{Chou1969PAMS}). We call a subset $A\subseteq \beta T_n$ \emph{right-invariant}\footnote{The reader is invited to compare this notion with Definition \ref{DefiInvariant} above. Notice that this notion is weaker since we only consider partial translations of $T_n$ given by adding a letter to the right, but not by deleting one.} if \[\tilde y(A)\subseteq A\ \text{ for all }\ y\in T_n.\]

 The following is a particular case of   \cite[Lemma 2 and Proposition 4.1]{Chou1969PAMS}.\footnote{Equivalently, this could be obtained as in Lemma \ref{LemmaSepHigCo} above.}

\begin{lemma}\label{LemmaChou}
 Given $n\in\N$,  $\beta T_n$ contains  at least $2^{2^{\aleph_0}}$ nonempty, mutually disjoint, closed, invariant subsets.\footnote{In \cite{Chou1969PAMS}, Chou works with semigroups, but this is precisely what $T_n$ is endowed with the products $x\ast y=x^\smallfrown y$.}
\end{lemma}

\begin{theorem}\label{ThmSemiG}
    In the setting of Assumption \ref{AssumptionBranchingTree}: If  $\beta=\log(n)$, then  there are $2^{2^{\aleph_0}}$ extreme  $(\sigma_h,\beta)$-KMS states on $\cstu(T_n)$.
\end{theorem}

Before proving Theorem \ref{ThmSemiG}, let us isolate an easy lemma for further reference. Precisely, the next result is simply a more specialized version of Theorem \ref{ThmSufficientKMS}. 

\begin{lemma}\label{LemmaSufficientKMSTn}
  In the setting of Assumption \ref{AssumptionBranchingTree}:  Suppose $\varphi$ is a state on $\ell_\infty(T_n)$ such that 
     \begin{equation}\label{EqVarphiSufficient9}
        \varphi(\chi_{\tilde y(A)} )
=\varphi\Big(\chi_A e^{\beta (\overline{ h-h\circ \tilde y})} \Big)
\end{equation}
for all  $y\in T_n$ and all $A\subseteq T_n$. Then,   $\varphi\circ E$ is a  $(\sigma_{h},\beta)$-KMS state on $\cstu(T_n)$; where $E\colon \cstu(T_n)\to \ell_\infty(T_n)$ is the canonical conditional expectation. 
\end{lemma}

\begin{proof}
Any partial isometry $f$ of $T_n$ can be written as $f=\bigcup_{i=1}^kf_i$ where each $f_i$ is a composition of partial isometries of the form \[x\in   A\to\tilde y(x)\in \tilde y(A)\] for some $y\in T_n$ and $A\subseteq T_n$, and partial isometries of the form \[\tilde y(x)\in \tilde y( A)\to x\in  A,\] for some $y\in T_n$ and $A\subseteq T_n$. Therefore, by Theorem \ref{ThmSufficientKMS},  it is enough to notice that \eqref{EqVarphiSufficient9} holds for partial isometries of the second kind. For that, fix $y\in T_n$ and $A\subseteq T_n$, and let $g\colon  \tilde y( A)\to  A$ be the partial isometry given by $g(\tilde y(x))=x$ for all $x\in A$. Then, since  
\[h(x)-h(\tilde y(x))=|y|\ \text{ for all }\ x\in T_n,\]
our assumption on $\varphi$ implies that 
\begin{align*}
    \varphi(\chi_A)&=e^{-\beta |y|}\varphi\Big(\chi_A e^{\beta (\overline{  h- h\circ \tilde y})}\Big)\\
    &=e^{-\beta|y|}\varphi(\chi_{\tilde y(A) })\\
    &=\varphi\Big(\chi_{\tilde y(A) }e^{\beta(\overline{  h- h\circ  g})}\Big).
\end{align*}
So, we are done.
\end{proof}

\begin{proof} [Proof of Theorem \ref{ThmSemiG}]
We start establishing some convention. Firstly, recall that $\ell_\infty(T_n)$ is canonically isomorphic to $C(\beta T_n)$. In order to keep track of notation, if $a\in \ell_\infty(T_n)$, we write $\bar a$ to denote $a$ as an element of $C(\beta T_n)$. Notice that, if $a=\chi_A$ for some $A\subseteq T_n$, then 
\[\overline{\chi_A}=\chi_{\bar A},\]
where the closure $\bar A$ is taken in $\beta T_n$.  Therefore, if $\varphi$ is a state on $\ell_\infty(T_n)$, we can view it as a state on $C(\beta T_n)$, i.e., $\varphi$ is a Borel measure on $\beta T_n$  and
\[\varphi(a)=\int_{\beta T_n}\bar ad \varphi\ \text{ for all }\ a\in \ell_\infty(T_n).\]
With this in mind, we define the \emph{support of $\varphi$} as the support of $\varphi$ as a Borel measure on $\beta T_n$ and   denote it by  $\supp(\varphi)\subseteq \beta T_n$. Suppose now that $\varphi$ is a state on $\cstu(T_n)$. Then, $\varphi\restriction \ell_\infty(T_n)$ is a state on $\ell_\infty(T_n)$ and, by abuse of notation, we   write
\[\supp(\varphi)=\supp(\varphi\restriction \ell_\infty(T_n)).\]

We now start the proof. By Lemma \ref{LemmaChou}, there is a family $(L_j)_{j\in J}$ of nonempty, mutually disjoint, closed, invariant subsets of $\beta T_n$ such that $|J|=2^{2^{\aleph_0}}$. Fix $j\in J$ and, for simplicity, let $L=L_j$.    Denote the subset of all $(\sigma_h,\beta)$-KMS states on $\cstu(T_n)$ which vanish on the compacts by $K_{\beta }$ and define 
\[K_{\beta }^L=\{\varphi \in K_{\beta }\mid \supp(\varphi)\subseteq L\}.\]
Clearly, $K_{\beta }^L$ is convex and weak$^*$-compact. Let us show $K_{\beta }^L$ is nonempty. 

By Theorem \ref{ThmnBranchingTree}, $K_\beta\neq \emptyset$. From now on, we fix $\varphi\in K_{\beta }$. As $L$ is nonempty, fix also $\omega\in L$. We define a state $\psi$ on $\ell_\infty(T_n)$ as follows: for each $a\in \ell_\infty(T_n)$, let  $\tilde a\in \ell_\infty(T_n)$ be given by 
\[\tilde a(y)=\bar a(\tilde y(\omega))\ \text{ for all }\ y\in T_n.\]
 We then let $\psi$ be the state on $\ell_\infty(T_n)$ given by 
\[\psi(a)=\varphi(\tilde a) \ \text{ for all }\ a\in\ell_\infty(T_n).\]
We extend $\psi$ to the whole $\cstu(T_n)$ in the usual way, that is, we let $\psi=\psi\circ E$ where $E\colon \cstu(T_n)\to \ell_\infty(T_n)$ is the canonical conditional expectation.
Since it is immediate that $\psi$ is indeed a state on $\cstu(T_n)$, we only need to show that $\psi$ satisfies the required KMS condition and that $\supp(\psi)\subseteq L$.

For the KMS conditions, let $y\in T_n$ and $A\subseteq T_n$;  so, $\tilde y\restriction A\colon A\to \tilde y(A)$ is a partial translation on $T_n$.  Notice that 
\begin{align}\label{EqChou1}
  \widetilde{\chi_{\tilde y(A)}}(x)= \overline{\chi_{\tilde y(A)}}(\tilde x(w))= \chi_{\overline{\tilde y(A)}}(\tilde x(w))=  \chi_{ \tilde y(\bar A)}(\tilde x(w)) 
\end{align}
for all $x\in T_n$. In order to understand $ \chi_{ \tilde y(\bar A)}(\tilde x(\omega)) $, notice that
\[\tilde y(\{x\in T_n\mid \tilde x(\omega)\in \bar A\})\subseteq \{x\in T_n\mid \tilde x(\omega)\in \tilde y(\bar A)\}\]
and
 \[\{x\in T_n\mid \tilde x(\omega)\in \tilde y(\bar A) \text{ and } |x|\geq |y|\}\subseteq \tilde y(\{x\in T_n\mid \tilde x(\omega)\in \bar A\}).\]
 Therefore, as $\{x\in T_n\mid |x|<|y|\}$ is finite and as $\varphi$ vanishes on compacts, letting  \[B=\{x\in T_n\mid \tilde x(\omega)\in \bar A\}\ \text{ and }\ C=\{x\in S\mid \tilde x(\omega)\in \tilde y(\bar A)\},\] we have that $\varphi(\chi_{\tilde y(B)})=\varphi(\chi_C)$. By \eqref{EqChou1}, we have   
 $\widetilde{\chi_{\tilde y(A)}}=\chi_C$ and our discussion gives
\begin{align}\label{Eq12345}
\psi(\chi_{\tilde y(A)})& =\varphi(\widetilde{\chi_{\tilde y(A)}})\\
&=\varphi(\chi_{\tilde y(B)})\notag
\\
&=\varphi\Big(\chi_B e^{\beta(\overline{ h- h\circ\tilde y})}
\Big) .\notag
\end{align}

As  $h- h\circ \tilde s$ is bounded,  $(  h- h\circ \tilde s)^\sim$ is well defined. Let $(z_j)_j$ be a net of elements of $T_n$ converging to $\omega$. Notice that    
\[h(x^\smallfrown y)=h(x)+h(y)\ \text{ for all }\ x,y\in T_n.\]
Therefore,  
\begin{align*}
(  h- h\circ \tilde y)^\sim(x)&=\overline{(  h- h\circ \tilde y)}(\tilde x(\omega))\\
&= \lim_{i}( h(z_j x)-h(z_jxy)) \\
&=\lim_{i}(  h(x)-h(xy))\\
&=(  h- h\circ \tilde y)(x)
\end{align*}
for all $x\in T_n$. By the definition of $B$, it is clear that $\chi_B=\widetilde{\chi_A}$. Therefore, 
\begin{align}\label{Eq123456}
    \psi\Big(\chi_Ae^{\beta(\overline{ h- h\circ\tilde y})}\Big)&=\varphi \Big(\Big(\chi_Ae^{\beta(\overline{h- h\circ\tilde y})}\Big)^\sim\Big)\\
    &=\varphi \Big(\widetilde{\chi_A}\Big(e^{\beta(\overline{ h- h\circ\tilde y})}\Big)^\sim\Big)\notag\\
    &=\varphi \Big(\chi_B e^{\beta(\overline{ h- h\circ\tilde y})} \Big).\notag
\end{align}
By \eqref{Eq12345} and \eqref{Eq123456}, we conclude that 
\[\psi(\chi_{\tilde y(A)})=\psi\Big(\chi_Ae^{\beta(\overline{ h- h\circ\tilde y})}\Big).\]
As $y\in T_n$ and $A\subseteq T_n$ were arbitrary, this shows that $\psi$ is a $(\sigma_h,\beta)$-KMS state on $\cstu(T_n)$.

Let us notice $\supp(\psi)\subseteq L$. Suppose $\omega'\not\in L$. Then there is $A\subseteq T_n $ such that $\omega'\in \bar A$ and $A\cap L=\emptyset$. As $\omega\in L$ and $L$ is invariant, $\tilde x(\omega)\in L$ for all $x\in T_n$. Hence, 
\[\widetilde{\chi_A}(x)=\chi_{\bar A}(\tilde x(\omega))=0\]
for all $x\in T_n$, i.e., $\widetilde{\chi_A}(x) $. Then, thinking of $\psi$ as being defined on $C(\beta T_n)$ as described above, we have that $\psi(\widetilde{\chi_A})=0$. This shows that $\supp(\psi)\subseteq L$ and we concluded our proof that $K^L_\beta\neq \emptyset$.

Since $j\in J$ was arbitrary, we have that each $K^{L_j}_\beta$ is convex, weak$^*$ compact, and nonempty. Hence,  Krein--Milman theorem implies that each of them contains extreme points. Since $(L_j)_{j\in J}$ are disjoint, this implies that there are $2^{2^{\aleph_0}}$ many extreme points and we are done.\end{proof}

 \subsubsection{Localization of KMS states on $\cstu(T_n)$ for $\beta=\log(n)$.}
We are left to notice that  a version of Theorem \ref{ThmSemiG} holds along every branch of $T_n$. For that, we must further analyze the Higson corona of $T_n$. More precisely, we must identify a \cstar-subalgebra of $C(\nu T_n)$ which will help us to locate the KMS states on $\cstu(T_n)$ for inverse temperature $\beta=\log(n)$ better.

We first introduce some notation. Firstly, let $[T_n]$ denote the \emph{branches} of $T_n$, i.e., \[[T_n]=\{1,\ldots, n\}^\N.\]
Given $\bar x=(x_j)_{j=1}^\infty\in [T_n]$ and $k\in\N$, we let $\bar x|k$ be the initial segment of $\bar x$ with $k$ letters, i.e., \[\bar x|k=(x_1,\ldots,x_k).\] 
 We now set
\[\mathcal T_n=T_n\cup [T_n]\] and endow  $\mathcal T_n$ with an appropriate topology. For that, we first extend the concatenation operation: for $y\in T_n$ and $\bar x\in [T_n]$, we let 
\[y^\smallfrown \bar x=(y_1,\ldots,y_{|y|}, x_1,x_2,\ldots)\in [T_n].\]
Given any  $y\in T_n$, we let
\[y^\smallfrown \mathcal T_n=\{x\in \mathcal T_n\mid \exists z\in \mathcal T_n\text{ with }x=y^\smallfrown z\},\]
i.e., $y^\smallfrown \mathcal T_n$   denotes the set of  words, finite or not, which ``start'' with $y$. We define $y^\smallfrown T_n$ and $y^\smallfrown [T_n]$ analogously, i.e., 
\[y^\smallfrown T_n=(y^\smallfrown \mathcal T_n)\cap T_n\ \text{ and }\ y^\smallfrown [T_n]=(y^\smallfrown \mathcal T_n)\cap[ T_n].\]

We endow  $\mathcal T_n$ with the topology generated by  
\[\cP(T_n)\cup\{y^\smallfrown \mathcal T_n\mid y\in T_n\}.\]
So, $T_n$ is an open subset of $\mathcal T_n$ and the inclusion \[T_n\hookrightarrow \mathcal T_n\] is a homeomorphic embedding with dense range. Moreover, it is easy to see that $\mathcal T_n$ is a compact space. Hence, $\mathcal T_n$ is a \emph{compactification} of $T_n$.

 As $T_n$ is dense in $\mathcal T_n$, this allow us to see $C(\mathcal T_n)$ as a \cstar-subalgebra of $\ell_\infty(T_n)$ in a canonical way. Precisely, we identify $C(\mathcal T_n)$ with the image of the following  injective $^*$-homomorphism
\[f\in C(\mathcal T_n)\mapsto f\restriction T_n\in \ell_\infty(T_n).\]

\begin{lemma}\label{LemmaHigsonCoronaTn}
Let $n\in\N$ and consider the $n$-branching tree $T_n$. Then:
\begin{enumerate}
\item\label{LemmaHigsonCoronaTn.Item1} For all $y\in T_n$, the projection $\chi_{y^\smallfrown T_n}$ is  a Higson function.
\item\label{LemmaHigsonCoronaTn.Item2} The Banach space 
\[C_n=\overline{\mathrm{span}}\{\chi_{y^\smallfrown T_n}\mid y\in T_n\}\]
is a \cstar-algebra contained in $C_h(T_n)$.
\item \label{LemmaHigsonCoronaTn.Item3} Under the identification  of   $C(\mathcal T_n)$ with the \cstar-subalgebra of $\ell_\infty(T_n)$ described above, 
 we have $C_n=C(\mathcal T_n)$. In particular, the compactification $\mathcal T_n$ is Higson compatible.
\end{enumerate}
In particular,   identifying $C([T_n])=C(\mathcal T_n)/c_0(T_n)$ via  Gelfand transfom, we have that $C([T_n])\subseteq \roeq(T_n)$. 
\end{lemma}

\begin{proof}
\eqref{LemmaHigsonCoronaTn.Item1}
Fix $y\in T_n$. Let $\eps>0$ and $R>0$. Let
\[F=\{x\in T_n\mid |x|\leq |y|+R\}.\]
Then, if $x,z\in T_n\setminus F$ and $d(x,z)<R$, we must have that either both $x$ and $z$ are in $y^\smallfrown T_n$, or neither of them are. In either case, we have \[|\chi_{y^\smallfrown T_n}(x)-\chi_{y^\smallfrown T_n}(z)|=0,\]
so $\chi_{y^\smallfrown T_n}$ is a Higson function.

\eqref{LemmaHigsonCoronaTn.Item2}
It is evident that $C_n$ is closed under the adjoint operator. So, we only need to show that $C_n$ is also closed under product. If $x,z\in T_n$, we write $x\leq z$ if $|x|\leq |z|$ and $x_i=z_i$ for all $i\in \{1,\ldots, |x|\}$. The fact that  $C$ is a \cstar-algebra follows from the straightforward fact that, for all $x,z\in T_n$, we have 
\[\chi_{y^\smallfrown T_n}\chi_{z^\smallfrown T_n}=
\left\{\begin{array}{ll}
\chi_{z^\smallfrown T_n},& \text{if }y\leq z,\\
\chi_{y^\smallfrown T_n},& \text{if }z\leq y,\\
0,& \text{otherwise.}
\end{array}\right.\]
So, $C_n$ is  closed under multiplication. The fact that $C_n\subseteq C_h(T_n)$ follows from \eqref{LemmaHigsonCoronaTn.Item1}. 

\eqref{LemmaHigsonCoronaTn.Item3} We start noticing that 
\begin{align*}
a\in C_n \ \Leftrightarrow\ & a=\sum_{y\in T_n}a_y\chi_{y^\smallfrown T_n} \text{ for some }(a_y)_{y\in T_n}\in \ell_\infty(T_n)\text{ such that }
\\
&\text{the sums } \Big(\sum_{k\in \N}a_{\bar x|k}\Big)_{\bar x\in [T_n]} \text{ are equi-convergent.}
\end{align*}
In particular, if $\bar x\in [T_n]$ and $a=\sum_{y\in T_n}a_y\chi_{y^\smallfrown T_n}$ is as above, the limit \[\lim_{k\to \infty}a(\bar x|k)=\sum_{k\in \N}a_{\bar x|k}\]
exists. We can then define an $^*$-isomorphic embedding $\Phi\colon C_n\to C(\mathcal T_n)$ by letting 
\[\Phi(a)(w)=\left\{\begin{array}{ll}
a(w),& \text{ if }w\in T_n,\\
\lim_{k\to \infty}a(w|k), & \text{ if }w\in [T_n].
\end{array}\right.\]
It is straightforward to show that $\Phi$ is indeed well-defined, i.e., $\Phi(a)$ is a continuous function on $\mathcal T_n$ for all $a\in C_n$. Moreover, it is also clear $\Phi$ is an injective $^*$-homomorphism and that 
\[\Phi(a)\restriction T_n=	a.\]

We are left to notice that the  $\Phi$ is sujective. For that, we show that the image of 
\[\mathrm{span}\{\chi_{y^\smallfrown T_n}\mid y\in T_n\}\]
under $\Phi$ is dense in $C(\mathcal T_n)$. Fix $f\in C(\mathcal T_n)$ and $\eps>0$. As $f$ is continuous and $[T_n]$ is compact, we can pick $y_1,\ldots, y_k\in T_n$ such that 
\begin{equation}\label{EqCoverBranches}[T_n]\subseteq \bigcup_{j=1}^ky_j^\smallfrown \mathcal T_n
\end{equation}
and
\[|f(x)-f(z)|<\eps\text{ for all }i\in \{1,\ldots, k\} \text{ and all } x,z\in y_j^\smallfrown T_n. \]
By \eqref{EqCoverBranches}, there is a finite set $F\subseteq T_n$ such that 
\[ T_n\subseteq F\cup \bigcup_{j=1}^ky_j^\smallfrown  T_n.\]
For simplicity, assume $F\cap y_j^\smallfrown  T_n=\emptyset$ for all $j\in \{1,\ldots,k\}$ and let  $a\in \ell_\infty(X)$ be given by 
\[a(x)=\left\{\begin{array}{ll}
f(x),& \text{ if }x\in F,\\
f(y_j),&\text{ if } j\in \{1,\ldots, k\}\text{ and } x\in y_j^\smallfrown T_n.
\end{array}\right.\]
It is straightforward to check that \[a\in\mathrm{span}\{\chi_{y^\smallfrown T_n}\mid y\in T_n\}\]
and that $\|\Phi(a)-f\|\leq \eps$. 
\end{proof}

The next couple of results will focus more on KMS states on $\roeq(T_n)$ and will not be necessary for the main result of this section per se (Theorem \ref{ThmnBranchingTreeCOMPLETEINTRO}). The reader interested only in Theorem \ref{ThmnBranchingTreeCOMPLETEINTRO} can safely skip to Lemma \ref{LemmaIsomTreeBranches}.

\begin{definition}
 In the setting of Assumption \ref{AssumptionBranchingTree}: For each $\beta>\log(n)$, let $\varphi_\beta$ be the $(\sigma_{h},\beta)$-KMS state in Theorem \ref{ThmnBranchingTree}. If $(\beta_k)_k\subseteq (\log(n),\infty)$ is a sequence converging to $\log(n)$ and $ \cU$ is a nonprincipal ultrafilter on $\N$, then
 \[\varphi=w^*\text{-}\lim_{k,\cU}\varphi_{\beta_k}\]
 is a $(\sigma_{h},\log(n))$-KMS state on $\cstu(T_n)$. We call any such KMS states  a \emph{limiting KMS state}. By   Theorem \ref{ThmnBranchingTree}, those states always vanish on $\cK(\ell_2(T_n))$.  
\end{definition}

\begin{corollary}\label{CorTnFaithfulQn}
 In the setting of Assumption \ref{AssumptionBranchingTree}: Let $\beta=\log(n)$ and  $\varphi$ be a limiting $(\sigma_{h},\beta)$-KMS state on $\cstu(T_n)$. Let $\psi$ be the   $(\sigma_{h}^\infty,\beta)$-KMS on $\roeq(T_n)$ such that   $\varphi=\psi\circ \pi$. Then, the rescriction of $\psi$ to $C([T_n])$ is faithful.
\end{corollary}

\begin{proof}
Let $\mu$ be the probability measure on $[T_n]$ given by Riesz representation theorem by restricting $\psi$ to $C([T_n])$, i.e., 
\[\psi(a)=\int_{[T_n]}ad\mu\ \text{ for all }\ a\in C([T_n]).\]

Since $\varphi$ is a limiting $(\sigma_{h},\log(n))$-KMS state, let $(\beta_k)_k\subseteq (\log(n),\infty)$ be a sequence converging to $\log(n)$ and $\cU$ be a nonprincipal ultrafilter such that 
 \[\varphi=w^*\text{-}\lim_{k,\cU}\varphi_{\beta_k}.\]
 By the formula of each $\varphi_{\beta_k}$ given by Theorem \ref{ThmnBranchingTree}, it follows that 
 \[\varphi_{\beta_k}(\chi_{y^\smallfrown T_n})=e^{-\beta_k|y|} \ \text{ for all }\ y\in T_n.\]
 Hence, by the formula of $\varphi$, we have 
 \[\varphi(\chi_{y^\smallfrown T_n})=\lim_{k\to \infty}e^{-\beta_k|y|}=\frac{1}{n^{|y|} }\ \text{ for all }\ y\in T_n.\]

 This shows that $\mu$ is the Bernoulli measure on $[T_n]=\{1,\ldots,n\}^\N$. Since the support of the   Bernoulli measure is the whole $[T_n]$, this shows that $\varphi$ is faithful on $C([T_n])$. This completes the proof.
\end{proof}

\begin{corollary}\label{CorManyKMSRoeq}
In the setting of Assumption \ref{AssumptionBranchingTree}: If $\beta=\log(n)$, then for all $\bar x\in [T_n]$ there is an extreme $(\sigma^\infty_h,\beta)$-KMS state $\psi$ on $\roeq(T_n)$ such that 
\[\varphi(a)=a(x)\ \text{ for all }\ a\in C([T_n]).\]
Moreover, if $\varphi$ is a limiting $(\sigma_{h},\beta)$-KMS state on $\cstu(T_n)$, then  the $(\sigma_{h}^\infty,\beta)$-KMS state $\psi$ on $\roeq(T_n)$ determined by $\varphi=\psi\circ \pi$ is not extreme.
\end{corollary}

\begin{proof}
The first assertion  follows from Theorem \ref{ThmExtRoeqGENERAL} and Corollary \ref{CorTnFaithfulQn}. For the second assertion, notice that if $\psi=\varphi\circ \pi$ were extreme, then there would be $\bar x\in [T_n]$ such that $\psi$ vanishes on the ideal 
\[J_x=\{a\in C([T_n])\mid a(\bar x)=0\}\]
(see Proposition \ref{PropRestrictCenterVanishes}). However, it was shown in the proof of Corollary \ref{CorTnFaithfulQn} that $\psi$ is   faithful on $C([T_n])$; contradiction.
\end{proof}

We now return to the proof of Theorem \ref{ThmnBranchingTreeCOMPLETEINTRO}.  The following lemma is trivial and we isolate it for further reference.

\begin{lemma}\label{LemmaIsomTreeBranches}
Let $n\in\N$ and $T_n$ be the $n$-branching tree. Given any $\bar x,\bar y\in [T_n]$ there is an isometry $f\colon T_n\to T_n$ such that  $f(\bar x|k)=\bar y|k$.\qed
\end{lemma}

 Given  a metric space $X$ and an isometry  $f\colon X\to X$, we let $u_f\colon \ell_2(X)\to \ell_2(X)$ be the (linear) isometry  determined by 
 \[u_f(\delta_x)=\delta_{f(x)}\ \text{ for all  }\ x\in X.\]
 
 \begin{lemma}\label{LemmaMovingKMSStatesOnTree}
Let $n\in\N$ and $T_n$ be the $n$-branching tree. Let $f\colon T_n\to T_n$ be an isometry  and consider the (linear) isometry $u_f\colon \ell_2(T_n)\to \ell_2(T_n)$ defined above. Then, the map 
\[\varphi\to \varphi\circ \mathrm{Ad}(u_f)\]
is an affine isometry of the set of $(\sigma_h,\beta)$-KMS states on $\cstu(T_n)$ to itself. \end{lemma}

\begin{proof}
It is enough to notice that $\varphi\circ \mathrm{Ad}(u_f)$ is a $(\sigma_h,\beta)$-KMS state on $\cstu(T_n)$ given that the same holds for $\varphi$. Indeed, once this is done the result follows since this map will clearly be an affine  isometry with inverse $\varphi\to \varphi\circ \mathrm{Ad}(u_f^*)$.

 Notice that  
\[h(f(x))=h(x)\ \text{  for all }x\in X.\]
Indeed,  any isometry of the tree $T_n$  must satisfy $f(\emptyset)=\emptyset$. Therefore, for each $x\in T_n$, we have
\[h(f(x))=d(f(x),\emptyset)=d(f(x),f(\emptyset))=d(x,\emptyset)=h(x).\]
Using this, an immediate computation gives us that 
\[\langle u_f^*e^{it\bar h}ae^{-it\bar h}u_f\delta_x,\delta_y\rangle=\langle e^{it\bar h}u_f^*au_fe^{-it\bar h}\delta_x,\delta_y\rangle\]
for all $t\in \R$, all $a\in \cstu(T_n)$, and all $x,y\in T_n$. In other words,     the flow $\sigma_h$ is invariant under $\mathrm{Ad}(u_f)$. This shows that  $\varphi\circ \mathrm{Ad}(u_f)$ must be a $(\sigma_h,\beta)$-KMS state on $\cstu(T_n)$ given that  $\varphi$ is one (equivalently, this could also be shown with the help of Theorem \ref{ThmSufficientKMS}).
\end{proof}

\begin{theorem}\label{ThmManyManyKMSStates}
In the setting of Assumption \ref{AssumptionBranchingTree}: If $\beta=\log(n)$, then for each $\bar x\in [T_n]$ there are $2^{2^{\aleph_0}}$ extreme $(\sigma_h,\beta)$-KMS states $\varphi$ on $\cstu(T_n)$  such that 
\[\varphi(\chi_{\bar x|k^\smallfrown T_n})=1\ \text{ for all }\ k\in\N.\]
\end{theorem}

\begin{proof}
Fix $\beta=\log(n)$. By Theorem \ref{ThmSemiG}, there are $2^{2^{\aleph_0}}$ extreme $(\sigma_h,\beta)$-KMS states on $\cstu(T_n)$.  By Lemma \ref{LemmaHigsonCoronaTn}, $\mathcal T_n$ is a Higson compatible compactification of $T_n$. Therefore, Theorem 
 \ref{ThmUltimoGENERAL} implies that for any extreme $(\sigma_h,\beta)$-KMS state $\varphi$ on $\cstu(T_n)$, there is   $\bar x\in [T_n]$ such that 
\begin{equation}\label{EqSupportedBranch1}\varphi(\chi_{\bar x|k^\smallfrown T_n})=1\ \text{ for all }\ k\in\N.
\end{equation}
Therefore, since $|[T_n]|=2^{\aleph_0}$, a pigeonhole argument implies that there is at least one   $\bar x\in [T_n]$ for which   there are $2^{2^{\aleph_0}}$ extreme $(\sigma_h,\beta)$-KMS states  on $\cstu(T_n)$ satisfying  \eqref{EqSupportedBranch1} for $\bar x$.
Fix such $\bar x\in [T_n]$. 

Let now $\bar y\in [T_n]$ be arbitrary and let $f\colon T_n\to T_n$ be an isometry such that \[f(\bar x|k)=\bar y|k\ \text{ for all }\ k\in\N \]
(Lemma \ref{LemmaIsomTreeBranches}). Clearly, we must have that 
\[f\big(\bar x|k^\smallfrown T_n\big)=\bar y|k^\smallfrown T_n \ \text{ for all }\ k\in\N.\]
Hence, $\mathrm{Ad}(u_f)(\chi_{\bar y|k^\smallfrown T_n})=\chi_{\bar x|k^\smallfrown T_n}$ for all $k\in\N$ and, if $\varphi$ satisfies \eqref{EqSupportedBranch1} for $\bar x$, it follows that
\[\big(\varphi  \circ\mathrm{Ad}(u_f)\big)(\chi_{\bar y|k^\smallfrown T_n})=1\ \text{ for all }\ k\in\N\]
The result then follows from Lemma \ref{LemmaMovingKMSStatesOnTree}.
\end{proof}
 
\begin{proof}[Proof of Theorem \ref{ThmnBranchingTreeCOMPLETEINTRO}]
 Theorem \ref{ThmnBranchingTree} gives that there is a $(\sigma_h,\beta)$-KMS state on $\cstu(T_n)$ if and only if $\beta\geq \log(n)$. Moreover, item \eqref{ThmnBranchingTreeItem2} and the first claim of item \eqref{ThmnBranchingTreeItem3}   of  Theorem \ref{ThmnBranchingTreeCOMPLETEINTRO} also follow from Theorem \ref{ThmnBranchingTree}.

We are left to notice that the second and third claim of   Theorem \ref{ThmnBranchingTreeCOMPLETEINTRO}\eqref{ThmnBranchingTreeItem3} hold. From now on, let $\beta=\log(n)$. By Lemma \ref{LemmaHigsonCoronaTn}, $\mathcal T_n$ is a Higson compatible compactification of $T_n$. Therefore, Theorem 
 \ref{ThmUltimoGENERAL} implies that any extreme $(\sigma_h,\beta)$-KMS state $\varphi$ on $\cstu(T_n)$ must have the required form, i.e., there must be $\bar x\in [T_n]$ such that 
 \[\varphi(\chi_{\bar x|k^\smallfrown T_n})=1 \ \text{ for all }\ k\in\N.\] Finally, the fact that for each $\bar x\in [T_n]$, there are $2^{2^{\aleph_0}}$ extreme $(\sigma_h,\beta)$-KMS states on $\cstu(T_n)$ satisfying the above is now simply  Theorem  \ref{ThmManyManyKMSStates}.
\end{proof}

\subsection{Obtaining distinct KMS states on $\cstu(T_n)$ for $\beta=\log(n)$}\label{SectionChaotic}

We finish the paper presenting a more concrete way of obtaining distinct KMS states for inverse temperature $\beta=\log(n)$. Precisely, if $(\beta_n)_n$ is a sequence converging to $\log(n)$ from the right and   $(\varphi_{\beta_n})_n$ is a sequence of states such that each $\varphi_{\beta_n}$ is a $(\sigma_h,\beta_n)_n$-KMS state on $\cstu(T_n)$, then $w^*\text{-}\lim_{n,\cU}\varphi_{\beta_n}$  is a $(\sigma_h, \log(n))$-KMS state, where $\cU$ is an arbitrary  nonprincipal ultrafilter on $\N$. The next theorem shows that, picking different sequences $(\beta_n)$ as above, this procedure may give us distinct  $(\sigma_h, \log(n))$-KMS states. As mentioned at the end of Subsection \ref{SubsectionIntroAppli}, this kind of behavior is unusual (see \cite{vanEnterRusze2007JStatPhy}) and known as \emph{chaotic behavior} of \emph{chaotic convergence of KMS states}.

\begin{theorem}\label{ThmRodrigoBissacot}
In the setting of Assumption \ref{AssumptionBranchingTree}: different sequences $(\beta_n)_n$ converging to $\log(n)$ may converge to distinct $(\sigma_h,\log(n))$-KMS states on $\cstu(T_n)$.
\end{theorem}

\begin{proof}
 Let $\beta=\log(n)$. For each $\beta'>\log(n)$,  let $\varphi_{\beta'} $ be the  $(\sigma_{h},\beta')$-KMS state on $\cstu(T_n)$ given by 
 Theorem \ref{ThmnBranchingTree}\eqref{ThmnBranchingTreeItem2}, i.e.,  
\[\varphi_{\beta' }([a_{x,y}])=\sum_{y\in T_n} a_{y,y}\Big(e^{-\beta|y|}-ne^{-\beta (|y|+1)}\Big)\]
for all $[a_{x,y}]\in \cstu(T_n)$. Given any nonprincipal ultrafilter $\cU$ on $\N$ and any sequence $(\beta_n)_n$ converging to $\beta$ from the right, we know that $w^*\text{-}\lim_{n,\cU}\varphi_{\beta_n}$ is a $(\sigma_{h},\beta)$-KMS state on $\cstu(T_n)$. Our strategy will be to  construct different  sequences $(\beta_n)_n$ as above which  give us different $(\sigma_{h},\beta)$-KMS states on $\cstu(T_n)$. For that, some manipulations with the formula of $\varphi_{\beta'}$ will be useful. Firstly, given  $E\subseteq T_n$ and  $k\in\N$, write
\[E_k=\{y\in E\mid |y|=k\}\] and notice that $|E_k|\leq  n^k$. Then, given an arbitrary $\beta'>\beta$, we have  
\begin{align*}\varphi_{\beta'}(\chi_{E}) & =\sum_{y\in E}(e^{-{\beta'}´|y|}-ne^{-{\beta'}´ (|y|+1)})\\
& = \sum_{k=0}^\infty |E_k|(e^{-{\beta'}´ k}-ne^{-{\beta'}´(k+1)})\\
&=(1-ne^{-{\beta'}})\sum_{k=0}^\infty |E_k|e^{-{\beta'} k}.
\end{align*}
Applying the change of variables $\tau=ne^{-{\beta'}}$ and letting $a_k=|E_k|/n^k$ for each $k\geq 0$, we have that each $a_k$ is in $[0,1]$ and 
\[\varphi_{\beta'}(\chi_{E})=(1-\tau)\sum_{k=0}^\infty a_k\tau^k.\]
Moreover, $\beta'\to \log(n)$ from the right if and only if $\tau\to 1$ from the left. At last, notice that if $E$ is such that there are $p<q\in \N$ with  
\[a_k=\left\{\begin{array}{ll}
1,& k\in  [p,q]\cap \N,\\
0,& k\not\in [p,q]\cap \N,
\end{array}\right.\]
then \begin{equation}\label{Eq.chiEtau}
    \varphi_{\beta'}(\chi_E)=\tau^p-\tau^{q+1}.
\end{equation}
This finishes the manipulations in the formula of $\varphi_{\beta'}$ that we will need.

We now construct increasing sequences $(\tau_k)_k$ and $(\theta_k)_k$ converging to $1$, and sequences $(p_k)_k$ and $(q_k)_k$ of natural numbers by induction for which the following holds 
\begin{itemize}
\item $p_k<q_k<p_{k+1}-1$ for all $k\in\N$,
\item $\tau_k^{p_k}-\tau_k^{q_k+1}>1/2$ for all $k\in\N$, and
\item  $\theta_k^{p_m}-\theta_k^{q_m+1}< 2^{-m-2}$ for all $k,m\in\N$.
\end{itemize}
This can be easily done as follows:  let $k\geq 2$ and suppose $(\tau_m)_{m=1}^{k-1}$, $(\theta_m)_{m=1}^{k-1}$, $(p_m)_{m=1}^{k-1}$, and $(q_m)_{m=1}^{k-1}$ where chosen appropriately; step $1$ of the induction can clearly be done.  Step $k$ of the induction goes as follows. Pick  $p_k>q_{k-1}+1$ such that  $\theta_m^{p_k}<2^{-k-2}$ for all $m\leq k-1$. Then pick $\tau_k\in (\tau_{k-1},1)$ such that $\tau_k^{p_k}>3/4$ and $q_k>p_k$ with $\tau_k^{q_k+1}<1/4$. Chose now $\theta_k\in (\theta_{k-1},1)$ with $\theta_k^{p_m}-\theta_k^{q_m+1}<2^{-m-2}$ for all $m\leq k$. This finishes the induction. 

We now use the sequences constructed in the previous paragraph to finish the proof. Precisely, we show that if $\cU$ is a nonprincipal ultrafilter on $\N$, then \[w^*\text{-}\lim_{n,\cU}\varphi_{\tau_n}\neq w^*\text{-}\lim_{n,\cU}\varphi_{\theta_n}.\]
For this, let $E\subseteq T_n$ be   given by 
\[E=\Big\{x\in T_n\mid |x|\in \bigcup_{m=1}^\infty [p_m,q_m]\Big\}.\]
Then   $|E_k|=n^k$ if $k\in \bigcup_{m=1}^\infty [p_m,q_m]$ and $|E_k|=0$ otherwise. Hence, letting $a_k=|E_k|/n^k$ as above, we have that 
\[a_k=\left\{\begin{array}{ll}
1,& k\in \bigcup_{m=1}^\infty [p_m,q_m] \\
 0,    & \text{otherwise.}
\end{array}\right.\]
 Therefore, using  \eqref{Eq.chiEtau} above, we have  
\[\varphi_{\tau_k}(\chi_E)=\sum_{m=1}^\infty (\tau_k^{p_m}-\tau_k^{q_m+1})\geq \tau_k^{p_k}-\tau_k^{q_k+1}>\frac{1}{2}\]
for all $k\in\N$. On the other hand, 
\[\varphi_{\theta_k}(\chi_E)=\sum_{m=1}^\infty (\theta_k^{p_m}-\theta_k^{q_m+1})<\sum_{m=1}^\infty 2^{-m-2}=1/4\]
for all $k\in\N$. Therefore, we conclude that 
\[\Big(w^*\text{-}\lim_{n,\cU}\varphi_{\tau_n}\Big)(\chi_E)\geq  1/2>1/4\geq \Big(w^*\text{-}\lim_{n,\cU}\varphi_{\theta_n}\Big)(\chi_E).\]
This finishes the proof.
\end{proof}

\begin{acknowledgments}
B. M. Braga would like to thank Alcides Buss and the Universidade Federal of Santa Catarina (UFSC) for an invitation which led him to meet Ruy Exel and, as a consequence,   made this project possible. R. Exel would like to express his thanks to IMPA for funding a two week  visit to Rio de Janeiro during which a large part of this project was developed. The authors are also thankful to Rodrigo Bissacot for bringing up to their attention the fact that chaotic convergence of KMS states is something yet far from being well understood. 
\end{acknowledgments}

\newcommand{\etalchar}[1]{$^{#1}$}
\providecommand{\bysame}{\leavevmode\hbox to3em{\hrulefill}\thinspace}
\providecommand{\MR}{\relax\ifhmode\unskip\space\fi MR }
\providecommand{\MRhref}[2]{%
  \href{http://www.ams.org/mathscinet-getitem?mr=#1}{#2}
}
\providecommand{\href}[2]{#2}

\end{document}